\documentclass[a4paper,10pt,reqno]{amsart}

\usepackage{amssymb,amsmath,amsthm,ascmac,array,bm,multirow}
\usepackage{booktabs}
\usepackage[dvipdfmx]{graphicx}
\usepackage{xcolor}
\usepackage{mathtools}
\usepackage{arydshln}
\mathtoolsset{showonlyrefs=true}

\usepackage[%dvipdfmx,
colorlinks=true,bookmarks=true,%citecolor=blue,
bookmarksnumbered=true,bookmarkstype=toc,linktocpage=true
]{hyperref}

\allowdisplaybreaks
\numberwithin{equation}{section}

\newtheorem{thm}{Theorem}[section]

\newtheorem{lem}[thm]{Lemma}

\newtheorem{cor}[thm]{Corollary}

\newtheorem{rem}[thm]{Remark}
\newtheorem{ass}[thm]{Assumption}

\newcommand{\mcc}{\mathcal{C}}
\newcommand{\mcf}{\mathcal{F}}

\newcommand{\mcl}{\mathcal{L}}

\newcommand{\mbbh}{\mathbb{H}}

\newcommand{\mbbr}{\mathbb{R}}

\newcommand{\mbby}{\mathbb{Y}}

\newcommand{\del}{\delta}

\newcommand{\sig}{\sigma}
\newcommand{\ep}{\epsilon}
\newcommand{\D}{\Delta}
\newcommand{\Sig}{\Sigma}
\newcommand{\lam}{\lambda}

\newcommand{\gam}{\gamma}
\newcommand{\Gam}{\Gamma}
\newcommand{\p}{\partial}

\newcommand{\cil}{\xrightarrow{\mcl}} % <- Convergence in law
\newcommand{\cip}{\xrightarrow{p}} % <- Convergence in probability
\newcommand{\argmin}{\mathop{\rm argmin}} 
\newcommand{\argmax}{\mathop{\rm argmax}}
\newcommand{\diag}{\mathop{\rm diag}} 
\newcommand{\tr}{\mathop{\rm Tr}} 

\def\ds#1{\displaystyle{#1}} % boldmath-->command: \Vec{...}

\def\nn{\nonumber}
\def\cadlag{c\`adl\`ag}

\def\lp{L\'{e}vy process}

\def\sumi{\sum_{i=1}^{N}}

\def\pr{P} \def\E{E}
\def\var{\mathrm{Var}}
\def\cov{\mathrm{Cov}}

\def\rev#1{\textcolor{black}{#1}}
\def\rrev#1{\textcolor{black}{#1}}

\newcommand{\tz}{\theta_{0}}
\newcommand{\tes}{\hat{\theta}_{N}}%\def\tes{\hat{\theta}_{n}}
\newcommand{\bes}{\hat{\beta}_{N}}

\newcommand{\ves}{\hat{v}_{N}}

\title[Gaussian quasi-likelihood analysis for mixed-effects models]
{Gaussian quasi-likelihood analysis for non-Gaussian linear mixed-effects model with system noise}

\author[T. Imamura]{Takumi Imamura}
    \address{Biostatistics Center, Shionogi \& Co., Ltd., Osaka, and Graduate School of Mathematical Sciences, University of Tokyo, Japan}

\author[H. Masuda]{Hiroki Masuda}
\address{Graduate School of Mathematical Sciences, University of Tokyo, 3-8-1 Komaba Meguro-ku Tokyo 153-8914, Japan}
\email{hmasuda@ms.u-tokyo.ac.jp}
\date{\today}

\begin{document}

\maketitle

\begin{abstract}
We consider statistical inference for a class of mixed-effects models with a system noise described by a non-Gaussian integrated Ornstein-Uhlenbeck process. Under the asymptotics where the number of individuals goes to infinity with possibly unbalanced sampling frequency across individuals, we prove some theoretical properties of the Gaussian quasi-likelihood function, followed by the asymptotic normality and the tail-probability estimate of the associated estimator. In addition to the joint inference, we propose and investigate the three-stage inference strategy, revealing that they are first-order equivalent while quantitatively different in the second-order terms.
Numerical experiments are given to illustrate the theoretical results.
% The result is then applied to derive model-selection criteria: the marginal and conditional Akaike information criteria and Schwarz's Bayesian information criterion.
\end{abstract}

%%%%%
%%%%%
\section{Introduction}

%%%%%%
\subsection{Background and motivation}
This paper aims to develop a statistical inference theory for a class of models used in longitudinal data analysis. Longitudinal data are repeated measurements or observations taken over time for multiple individuals; for example, in HIV research, the CD4 lymphocyte count and the HIV viral load. In these longitudinal data analyzes, our objective is to infer or evaluate changes over time in the mean structure of the response variable, the effects of covariates on the response variable, and the within-individual correlations of the response variable.
\rrev{Theoretically, by fitting a normal distribution through targeting only the mean and covariance structures of the random dynamics, our primary focus is to propose an explicit and easy-to-use estimation method, and to derive its theoretical properties.}

When longitudinal data are measured or collected prospectively, the time points at which the data will be measured are usually set in advance. However, due to reasons such as dropout from the longitudinal study, not all individuals are necessarily measured at all planned time points. In such cases, the number of measurements may vary between individuals and the measurement intervals within and between individuals may also vary. Such is called the ``unbalanced'' data set.
As a traditional approach to handling the unbalanced data set, linear mixed-effects (LME) models \cite{LairdWare1982} are frequently used. As an alternative approach, LME models with a Gaussian integrated Ornstein-Uhlenbeck (OU) process as the system noise are proposed in \cite{TayCumSy94}; see also \cite{HugKenSteTil17}. A special feature of this model is that we can estimate the degree of derivative tracking from longitudinal data \cite{BosTayLaw1998}. We consider that the trajectory of each individual tends to follow a linear path. In that case, the model is said to have strong derivative tracking (i.e., a LME model in which explanatory variables for fixed and random effects include time variables). On the other hand, if the slope of each individual's trajectory tends to change continuously, the model is said to have weak derivative tracking. See \cite{TayCumSy94} for more details on the derivative tracking.

The previous study \cite{ImaMasTaj24} showed the local asymptotic normality and the optimality of a local maximum-likelihood estimator for a class of Gaussian LME models with the integrated OU process as the system noise. Although the classical LME models are usually applied under the Gaussianity of the random effect and the measurement error, there have been some studies about model misspecification of the random effect in the context of (generalized) LME models, e.g. \cite{McCNeu2011} and \cite{McCNeu11}. 
\rev{Furthermore, the previous study \cite{asar2020linear} suggested through simulations and case studies that using models that include non-Gaussian system-noise can lead to reliable inferences about fixed-effect parameters. Based on these previous studies, it will be practically useful to develop a theory of inference that does not require the assumption that the random-effect, the measurement error, and the system-noise are Gaussian. In such cases where Gaussianity is not assumed, it is important to consider what type of likelihood function should be used to make statistical inferences. 
\rrev{
While an ideal likelihood function can be derived from empirical evidence and data characteristics, such a specification is not always feasible. In such cases, it is preferable to employ a simple yet distributionally robust likelihood function that retains theoretical validity even under model misspecification.
}
As a simple likelihood function, we consider an estimation methodology using the Gaussian quasi-likelihood random function (GQLF). 
If the estimator obtained from the GQLF exhibits desirable asymptotic behavior (e.g. consistency and asymptotic normality), it is expected that analysts can make statistical inferences without using a complex likelihood function;
this is exactly the primary objective of this paper.
%% HM added -->
For the driving L\'{e}vy process, we will only impose some moment conditions without specifying any distributional class, so that the proposed method can apply to a wide range of models.
%% <-- HM added
}

In this paper, we consider a class of LME models with the possibly non-Gaussian integrated L\'{e}vy-driven OU process as the system noise. On the one hand, as in \cite{TayCumSy94}, thanks to the continuous-time framework, this framework allows us to smoothly handle unbalanced longitudinal data sets in a unified manner; this nice feature cannot hold for the discrete-time first-order autoregressive structure. 
On the other hand, by adding the integrated OU process term, the Gaussian quasi-likelihood function becomes nonlinear for parameters associated with the OU process, raising concern about the large computational load of simultaneous optimizations for parameter estimates \cite{ImaMasTaj24}. To mitigate this problem, we propose a three-stage stepwise inference strategy in which the mean and covariance structures are optimized separately and alternately. 
% \rev{This stepwise inference is possible because the GQLF has no overlapping parameters in its mean and covariance structures.} 
By splitting the target parameters, it is expected that the computational load will be reduced compared to simultaneous optimization.

In our main result, we will show the very strong mode of convergence of the quasi-likelihood-ratio random field, namely, not only the weak convergence (locally asymptotically quadratic property) and uniform tail-probability estimate. To the best of our knowledge, within the class of LME models, there has been no previous study that compared joint likelihood inference with stepwise likelihood inference in terms of computational load and theoretical properties. 

%%%%%%
\subsection{Setup and objective}
\label{hm:sec_basic.setup}
Suppose that we are given a longitudinal data set from $i$th individual at given time points $0=t_{i0}<t_{i1}<\dots<t_{i n_i}$, described by
\begin{align}
Y_i(t_{ij}) &= X_i(t_{ij})^\top \beta + Z_i(t_{ij})^\top b_i + W_i(t_{ij}) +\ep_i(t_{ij})
%&= Y_i^\ast(t_{ij}) + \ep_{i}(t_{ij})
\label{hm:target.model}
\end{align}
for $1\le i\le N$ and $1\le j\le n_i$, where $\top$ denotes the transposition of a matrix and
\begin{equation}
\max_{i\le N}n_i = O(1).
\label{hm:ni-order}
\end{equation}
Here and in what follows, the asymptotics are taken for $N\to\infty$.
We will use the generic convention $\xi_{ij}=\xi_i(t_{ij})$, so that \eqref{hm:target.model} becomes
\begin{equation}
Y_{ij} = X_{ij}^\top \beta + Z_{ij}^\top b_i + W_{ij} +\ep_{ij}.
\label{hm:model_Yij}
\end{equation}
The ingredients are specified as follows.
\begin{itemize}
\item $X_{ij}\in\mbbr^{p_\beta}$ and $Z_{ij}\in\mbbr^{p_b}$ denote non-random explanatory variables for fixed and random effects of the $i$th individual, respectively, such that
\begin{equation}\nn
    \sup_{N}\max_{i\le N} \left(|X_i|+|Z_i|\right) < \infty,
\end{equation}
where, with a slight abuse of notation, $X_i :=(X_{ij})_{j=1}^{n_i} \in \mbbr^{n_i}\otimes\mbbr^{p_\beta}$ and $Z_i:=(Z_{ij})_{j=1}^{n_i} \in \mbbr^{n_i}\otimes\mbbr^{p_b}$, and $|\cdot|$ denotes the Euclidean norm.

\item $\beta\in\mbbr^{p_\beta}$ is the unknown fixed-effect parameter, which is common across the individuals.

\item Let $b_1,b_2,\dots$ be unobserved random-effects that are i.i.d. zero-mean random variables in $\mbbr^{p_b}$ with common nonnegative-definite covariance matrix $\Psi(\gam)$ for some function $\Psi:\,\mbbr^\gam \to \mbbr^{p_b}\otimes\mbbr^{p_b}$.
\begin{itemize}
    \item We do not fully specify the common distribution $\mcl(b_1)$. For a specific form of $\Psi(\gam)$, one may adopt the unstructured setting where all the entries of $\Psi(\gam)$ are fully unknown. However, it may suffer from computational issues caused by the high dimensionality of the parameters.
\end{itemize}

\item The stochastic process $W_i(\cdot)$ represents an unobserved random system-noise process driving the $i$th individual ($i=1,\dots, N$), described as the integrated Ornstein-Uhlenbeck (intOU) process:
\begin{equation}\nn
   W_i(t)=\int_0^t \zeta_i(s)ds, 
\end{equation}
% $W_i(t)=\int_0^t \zeta_i(s)ds$ 
where $\zeta_i(\cdot)$ denotes the L\'{e}vy-driven OU process with the autoregression parameter $\lam>0$ and the scale coefficient $\sig>0$; see \eqref{hm:ou_def} below. 
%More details are given below.

\item The processes $\ep_1(\cdot),\ep_2(\cdot),\dots$ denote i.i.d. white noise process representing measurement error: for each $i$, the variables $\ep_i(t_{i1}),\dots,\ep_i(t_{i n_i})$ are centered and uncorrelated, and have variance $\sig_\ep^2$.

\item The random variables $\{b_i\}$, $\{W_i(\cdot)\}$, and $\{\epsilon_i\}$ are mutually independent.

\end{itemize}
All the random elements introduced above are defined on an underlying filtered probability space 
% $(\Omega,\mcf,(\mcf_t)_{t\le T},\pr)$ 
endowed with the i.i.d. random sequence 
% $\{(b_i,\zeta_i(0),(L_i(t))_{t\le T},(\ep_i(t_{ij}))_{j\le n_i})\}_{i\ge 1}$, 
\begin{equation}\nn
    \{(b_i,\zeta_i(0),(L_i(t))_{t\le T},(\ep_i(t_{ij}))_{j\le n_i})\}_{i\ge 1},
\end{equation}
where $T>0$ is a fixed number for which $\sup_{N\ge 1}\max_{i\le N} \max_{j\le n_i} t_{ij} \le T$ (such a $T$ does exist under \eqref{hm:ni-order}).
As before, we will simply write the response-variable vectors $Y_i :=(Y_{ij})_{j=1}^{n_i} \in \mbbr^{n_i}$, $1\le i\le N$, so that
\begin{equation}
Y_{i} = X_{i}\beta + Z_{i} b_i + W_{i} +\ep_{i} \nn
\end{equation}
in the matrix-product form.

The model for the observation $\{(X_i,Y_i,Z_i)\}_{i\le N}$ is thus indexed by the finite-dimensional parameter
\begin{align}
\theta := (\beta,v)=\left(\beta, \gam, \lam,\sig^2,\sig_\ep^2 \right) \in \Theta 
&= \Theta_\beta \times \Theta_v = \Theta_\beta \times \Theta_\gam 
\times \Theta_{\lam} \times \Theta_{\sig^2} \times \Theta_{\sig_{\epsilon}^2} 
\nn\\
&\subset \mbbr^{p_\beta}\times\mbbr^{p_\gam}\times(0,\infty)\times(0,\infty)\times(0,\infty)
\nonumber
\end{align}
with $v := (\gam, \lam, \sig^2, \sig_{\epsilon}^2)$ denoting the covariance parameter. We assume that the parameter space $\Theta$ is a bounded convex domain in $\mbbr^p$ with $p := p_\beta + p_v$, where $p_v := p_\gam + 3$ denotes the dimension of $v$.
Throughout, we fix a point $\tz=(\beta_0,v_0)\in\Theta$ as a true value of $\theta$, assumed to exist.
It should be noted that the parameter may not completely characterize the distribution of the model, for we do not fully specify the distributions of $b_i$, $\zeta_i$, and $\ep_i$; in this sense, the model is semiparametric. We will denote by $\pr_\theta$, $\E_\theta$, $\var_\theta$, and $\cov_\theta$ the corresponding probability, expectation, variance, and covariance, respectively.
The subscript ``$\tz$'' will be omitted such as $\pr=\pr_{\tz}$ and $\E=\E_{\tz}$.

For convenience, we briefly mention some preliminary facts about the intOU process $W_i(\cdot)$.
Let $\zeta_1(\cdot),\zeta_2(\cdot),\dots$ be i.i.d. OU processes given by the stochastic differential equation
\begin{equation}
d\zeta_i(t)=-\lam \zeta_i(t) dt+\sig dL_i(t)
\label{hm:ou_def}
\end{equation}
for $i=1,\dots,N$, where $L_1,L_2,\dots$ are i.i.d. {\cadlag} {\lp es} such that $E_\theta[L_i(t)]=0$ and $\var_\theta[L_i(t)]=t$ for $t\in[0,T]$. The process $\zeta_i$ has a unique invariant distribution. We assume that $\zeta_i$ are strictly stationary, that is, $\zeta_i(0)$ obeys the invariant distribution; in this case, we can write
\begin{equation*}
\zeta_i(t) =\int_{-\infty}^t e^{-\lam(t-s)}\sig dL_i(s)    
\end{equation*}
% $\zeta_i(t) =\int_{-\infty}^t e^{-\lam(t-s)}\sig dL_i(s)$ 
for a two-sided version $(L_i(t))_{t\in\mbbr}$ of $L_i$.
We know that each $\zeta_i$ is exponentially ergodic. We refer to \cite{Mas04}, \cite{Mas07}, and the references therein for related details.
Now, we define $W_i(t)$ as unobserved random system-noise processes driving the $i$th individual ($i=1,\dots,N$), described as the intOU process: 
\begin{align}
W_i(t) & :=\int_0^t\zeta_i(s) ds \nn\\
&= \int_0^t\left(e^{-\lam s}\zeta_i(0) + \sig \int_0^s e^{-\lam(s-v)} dL_i(v)\right)ds \nn\\
% &= \frac{\zeta_i(0)}{\lam}(1-e^{-\lam t}) + \sig \int_0^t \int_0^s e^{-\lam(s-v)}dL_i(v) ds \nn\\
% &= \frac{\zeta_i(0)}{\lam}(1-e^{-\lam t}) + \sig \int_0^t \int_u^t e^{-\lam(s-v)} ds dL_i(v) \quad \text{(Stochastic Fubini)} \nn\\
&= \frac{\zeta_i(0)}{\lam}(1-e^{-\lam t}) + \frac{\sig}{\lam} \int_0^t (1-e^{-\lam(t-v)}) dL_i(v).
\nonumber
\end{align}
We denote by 
% $H_i(\lam, \sig^2) := \cov[W_i,W_i]=(\cov[W_{ij}, W_{ik}])_{j, k=1}^{n_i}$ 
$H_i(\lam, \sig^2) = (H_{i;jk}(\lam, \sig^2))_{j,k=1}^{n_i} \in \mbbr^{n_i}\otimes\mbbr^{n_i}$ 
the covariance matrix of $W_i :=(W_{ij})_{j=1}^{n_i}$:
\begin{align}
H_{i;jk}(\lam,\sig^2) 
&:= \cov_\theta\left[W_{ij}, W_{ik}\right]\nn \\
% &=\frac{1}{\lam^2}(1-e^{-\lam t_{ij}})(1-e^{-\lam t_{ik}}) \E_\theta[\zeta_i(0)^2] \nn\\
% &{}\qquad + \frac{\sig^2}{\lam^2}\int_0^{t_{ij}\wedge t_{ik}} (1-e^{-\lam (t_{ij}-s)})(1-e^{-\lam (t_{ik}-s)}) ds
% \nn\\
&=\frac{\sig^2}{2\lam^3}\left(2\lam \min(t_{ij},t_{ik})+e^{-\lam t_{ij}}+e^{-\lam t_{ik}}-1-e^{-\lam |t_{ij}-t_{ik}|}\right).
\label{hm:H_form}
\end{align}

Under the aforementioned setup, our objective in this paper is to investigate the asymptotic behavior of the marginal GQLF, based on which we can prove the asymptotic normality, the second-order asymptotic expansion, and the tail-probability estimate of the associated estimator. The GQLF provides us with an explicit inference strategy only by using the second-order (covariance) structure without full distributional specification of the underlying model. We will formulate the two kinds of GQLFs, the joint and the stepwise ones.
Although our primary interest is the intOU mixed-effects model \eqref{hm:model_Yij}, in the main sections \ref{hm:sec_joint} and \ref{hm:sec_stepwise} we will work with the following notation for the mean vector and the covariance matrix:
\begin{align}
\mu_i(\beta) &:= \E_\theta[Y_{i}] = X_i \beta, \nn\\
\Sig_i(v) &:= \cov_\theta[Y_{i}] = Z_i \Psi(\gam) Z_i^\top + H_i(\lam, \sig^2) + \sig_{\epsilon}^2 I_{n_i}.
\label{hm:Sig_def}
\end{align}
This will not only make the arguments more convenient and transparent but also make extensions to various non-linear settings straightforward.

%%%
\subsection{Outline}

% The rest of this paper is organized as follows. 
We first study the joint GQLF in Section \ref{hm:sec_joint} and then the stepwise GQLF in Section \ref{hm:sec_stepwise}. In both cases, we obtain the asymptotic normality, the second-order stochastic expansion of the estimator, and the tail-probability estimate.
Section \ref{hm:sec_linear.case} provides some remarks about the original setup \eqref{hm:model_Yij}. Section \ref{sec:simulations} presents some illustrative simulation results and data analysis.
%Finally, Section \ref{ti:sec_stepwise.proof} gives the proofs concerning the asymptotics of the stepwise GQLF.

%%%%%
\subsection{Comments on model selection}

Our results include the asymptotic normality at rate $\sqrt{N}$ of the form 
$\sqrt{N}(\tes-\tz) \cil N_p\left(0,\, \Gam_0^{-1} S_0 \Gam_0^{-1}\right)$ 
% \begin{equation}\nn
%     \sqrt{N}(\tes-\tz) \cil N_p\left(0,\, \Gam_0^{-1} S_0 \Gam_0^{-1}\right)
% \end{equation}
and the tail-probability estimate $\sup_N \pr[|\sqrt{N}(\tes-\tz)|>r]\lesssim r^{-L}$ for any $L>0$.
With these results, it is routine to derive the fundamental model selection criteria associated with the joint GQLF $\mbbh_N(\theta)$:
the classical marginal Akaike information criteria (AIC)
\begin{equation}\nn
    -2\mbbh_N(\tes) + 2\tr(\hat{\Gam}_N^{\prime -1} \hat{S}'_N),
\end{equation}
where $\hat{\Gam}'_N$ and $\hat{S}'_N$ are suitable consistent estimators of $\Gam_0$ and $S_0$, respectively, 
and Schwarz's Bayesian information criterion (BIC)
\begin{equation}
\label{ti:BIC}
 -2\mbbh_N(\tes) + p \log N.
\end{equation}
Concerning this point, we refer to \cite{EguMas18} and \cite{EguMas24} for detailed studies of AIC- and BIC-type statistics based on the GQLF.
% for stochastic differential equation models;  although the corresponding technicalities are essentially different, they have the same concept and the flow of the derivations is more or less the same.

Yet another well-known information criterion is the conditional AIC (cAIC) introduced in \cite{VaiBla05}; see also \cite{Kub11} and \cite{MulSceWel13} for some details of a conditional AIC based on the genuine likelihood. Formulating and deriving the cAIC will require different considerations, and we hope to report it elsewhere.

\subsection{Basic notation}
% Let $C>0$ denote a universal positive constant which may change at each appearance.
For two real positive sequences $(a_N)$ and $(b_N)$, we write $a_N \lesssim b_N$ if $\limsup_{N}(a_N/b_N) < \infty$.
We use the multilinear-form notation 
% $M[u_{i_1}, \dots, u_{i_m}] := \sum_{i_1 \dots i_m} M_{i_1 \dots i_m} u_{i_1}\dots u_{i_m}$ 
\begin{equation}\nn
    M[u_{i_1}, \dots, u_{i_m}] := \sum_{i_1 \dots i_m} M_{i_1 \dots i_m} u_{i_1}\dots u_{i_m}
\end{equation}
for a tensor $M = \{M_{i_1 \dots i_m}\}$; it may take values in a multilinear form.
For a square matrix $A$, we denote its 
%transposition, 
Frobenius norm by $|A|$, minimum eigenvalue by $\lambda_{\min}(A)$, and trace by $\tr(A)$.
% $A^\top$, 
The $d$-dimensional identity matrix is denoted by $I_d$.
The $k$th partial differentiation operator with respect to variables $a$ is denoted by $\p^k_{a}$, with $\p_a$ for $k=1$.
We use the symbol $\phi_{n_i}(\cdot;\mu,\Sig)$ for the $n_i$-dimensional Gaussian $N_{n_i}(\mu,\Sig)$-density.

\rev{
% Since there are many notations in this paper, 
For convenience of reference, in Table \ref{List_notation} we provide a list of main notations used in this paper.
% that have appeared so far and will appear hereafter.
}

\begin{footnotesize}
\begin{table}[h]
\centering
\caption{\rev{Brief notation table}}
\label{List_notation}
\begin{tabular}{p{1.7cm}lp{4.5cm}}
\toprule
{} & \rev{Notation} & \rev{Description} \\
\hline 
\rev{Observed data for $i$th individual} & \rev{$Y_i := (Y_{ij})_{j=1}^{n_i} \in \mbbr^{n_i}$} &\rev{all response data} \\
{} & \rev{$X_i := (X_{ij})_{j=1}^{n_i} \in \mbbr^{n_i}\otimes\mbbr^{p_\beta}$} & \rev{all explanatory data for fixed-effect} \\
{} & \rev{$Z_i := (Z_{ij})_{j=1}^{n_i} \in \mbbr^{n_i}\otimes\mbbr^{p_b}$} & \rev{all explanatory data for random-effects} \\
\hline

\rev{Random variable for $i$th individual} & \rev{$b_i \in \mbbr^{p_b}$} & \rev{random-effect parameter} \\
{} & \rev{$W_i(t) \in \mbbr$} & \rev{system-noise process} \\
{} & \rev{$\ep_i := (\ep_{ij})_{j=1}^{n_i} \in \mbbr^{n_i}$} & \rev{all measurement errors} \\
\hline

\rev{Parameter} & \rev{$\beta \in \mbbr^{p_\beta}$} & \rev{fixed-effect parameter} \\
{} & \rev{$\gam \in \mbbr^{p_\gam}$} & \rev{parameter that composes the covariance matrix of random-effect parameter} \\
{} & \rev{$\lam \in (0,\infty)$} & \rev{autoregression parameter of system-noise process} \\
{} & \rev{$\sig \in (0,\infty)$} & \rev{scale parameter of system-noise process} \\
{} & \rev{$\sig_{\ep} \in (0,\infty)$} & \rev{variance parameter of measurement error} \\
{} & \rev{$v := (\gam,\lam,\sig^2,\sig_{\ep}^2)$} & \rev{$p_v$-dimensional covariance parameter for response variable} \\
{} & \rev{$\theta := (\beta, v)$} & \rev{$p$-dimensional vector that bundles all parameters} \\
{} & \rev{$\tz = (\beta_0, v_0)$} & \rev{true value of $\theta$} \\
{} & \rev{$\tes = (\bes, \ves)$} & \rev{joint GQMLE} \\
{} & \rev{$\widetilde{\theta}_N = (\widetilde{\beta}_N, \widetilde{v}_N)$} & \rev{stepwise GQMLE} \\
{} & \rev{$\widetilde{\beta}_{N,1} \in \mbbr^{p_\beta}$} & \rev{stepwise GQMLE in Stage 1} \\
\hline

\rev{Random function} &\rev{$\mbbh_N(\theta) := \sumi \log\phi_{n_i}\left(Y_i;\,\mu_i(\beta),\Sig_i(v)\right)$} & \rev{joint GQLF} \\
{} & \rev{$\mbby_N(\theta) := \frac{1}{N}(\mbbh_N(\theta)-\mbbh_N(\theta_0))$} & \rev{joint quasi-Kullback-Leibler divergence} \\
{} & \rev{$\D_N(\theta) := \frac{1}{\sqrt{N}}\p_\theta \mbbh_N(\theta)$} & \rev{\textbf{joint quasi-score function}} \\
%{} & \rev{$S_N := \cov[\D_N(\theta_0)]$} & \rev{covariance matrix of the joint quasi-score function for $\tz$} \\
{} & \rev{$\Gam_N(\theta) := -\frac{1}{N}\p_\theta^2\mbbh_N(\theta)$} & \rev{\textbf{joint quasi-observed information matrix}} \\
{} & \rev{$\mbbh_{N,(1)}(\beta) := \sumi \log\phi_{n_i}\left(Y_i;\,\mu_i(\beta),I_{n_i}\right)$} & \rev{stepwise GQLF in Stage 1} \\
{} & \rev{$\mbbh_{N,(2)}(v) := \mbbh_N(\widetilde{\beta}_{N,1},v)$} & \rev{stepwise GQLF in Stage 2} \\
{} & \rev{$\mbbh_{N,(3)}(\beta) := \mbbh_N(\beta,\widetilde{v}_N)$} & \rev{stepwise GQLF in Stage 3} \\
{} & \rev{$\D_{N,(k)}$} & \rev{\textbf{stepwise quasi-score function} in Stage $k$ for $\beta_0$ in Stage 1 and 3 and $v_0$ in Stage 2} \\
{} & \rev{$\Gam_{N,(k)}$} & \rev{\textbf{stepwise quasi-observed information matrix} in Stage $k$ for $\beta_0$ in Stage 1 and 3 and $v_0$ in Stage 2} \\
\hline

\rev{Other function} & \rev{$\mu_i(\beta) \in \mbbr^{n_i}$} & \rev{mean vector of the $i$th individual} \\
{} & \rev{$\Sig_i(v) \in \mbbr^{n_i}\otimes \mbbr^{n_i}$} & \rev{covariance matrix of the $i$th individual} \\
{} & \rev{$\Psi(\gam) \in \mbbr^{p_b}\otimes \mbbr^{p_b}$} & \rev{covariance matrix of random-effect parameter} \\
{} & \rev{$H_i(\lam,\sig) \in \mbbr^{n_i} \otimes\mbbr^{n_i}$} & \rev{covariance matrix of system noise} \\

\hline
\end{tabular}
\end{table}
\end{footnotesize}

%%%%%
%%%%%
\section{Joint Gaussian quasi-likelihood analysis}
\label{hm:sec_joint}

% %%%%%
% \subsection{Joint Gaussian quasi-likelihood}

The joint GQLF is defined by
\begin{equation}
\mbbh_N(\theta)=\mbbh_N(\beta, v) := \sumi \log\phi_{n_i}\left(Y_i;\,\mu_i(\beta),\Sig_i(v)\right).
\nonumber
\end{equation}
%For an $\mcf$-measurable random variable $(\mu,\Sigma)\in\mbbr^{p}\times\mbbr^{p\times p}$, we denote by $N_{p}(\mu,\Sigma)$ the $\mcf$-conditional Gaussian distribution associated with the characteristic function $\xi\mapsto\E[\exp\{i\mu[\xi]-(1/2)\Sigma [\xi,\xi] \}]$.
Although the data generating distribution $\mcl(Y_{i})$ may not be Gaussian, we set our statistical model Gaussian with possibly different dimensions across the indices $i=1,\dots, N$; of course, $\mbbh_N(\theta)$ is the exact log-likelihood if $Y_1,\dots, Y_N$ are truly Gaussian.

In addition to the standing assumptions described in Section \ref{hm:sec_basic.setup}, we impose further regularity conditions. Denote by $\overline{\Theta}$ the closure of $\Theta$.

\begin{ass}\label{hm:A_1}
~
\begin{enumerate}
    \item The functions $\theta \mapsto(\mu_i(\beta),\Sig_i(v))$ $(i\ge 1)$ are of class $\mcc^{4}(\Theta)$ and all the derivatives with itself are continuous in $\theta\in\overline{\Theta}$.
    \item $\ds{\inf_N \min_{1 \leq i \leq N} \inf_{v \in \Theta_v} \lambda_{\min}(\Sigma_i(v)) > 0}$.
    \item 
    $\ds{\sup_{N\ge 1} \max_{1 \leq i \leq N} \max_{1\le k\le 4} \left(\sup_{\beta \in \Theta_{\beta}} |\p_{\beta}^k \mu_i(\beta)| \vee \sup_{v \in \Theta_v} |\p_{v}^k \Sigma_i(v)| \right) < \infty}$.
%     \item In addition to \eqref{hm:ni-order}, the sequence $n_1,n_2,\dots$ satisfies that
% \begin{equation}
%     \limsup_N \left|N^{\kappa}\left(\frac1N \sumi n_i - \overline{n}\right)\right| < \infty
%     \nn%\label{hm:ni-order-2}
% \end{equation}
%     for some real numbers $\kappa>0$ and $\overline{n}\ge 0$.
\end{enumerate}
\end{ass}

\begin{ass}\label{hm:A_2}
$\E[|L_1(t)|^q]+\E[|\ep_1(t)|^q]+\E[|b_1|^q]<\infty$ for every $q>0$ and $t\le T$.
\end{ass}

The joint \textit{Gaussian quasi-maximum likelihood estimator (GQMLE)} is defined to be any element
\begin{equation}
    \tes=(\bes,\ves) \in \argmax_{\theta \in \overline{\Theta}} \mathbb{H}_{N}(\theta).
    \nn
\end{equation} 
Under Assumption \ref{hm:A_2}, at least one such $\tes$ does exist ($\pr$-)a.s.

%%%%%
\subsection{Uniform convergence of quasi-Kullback-Leibler divergence}

To deduce the consistency of the joint GQMLE, we will prove the asymptotic behavior of the normalized \textit{quasi-Kullback-Leibler divergence} associated with $\mbbh_N$, defined by
\begin{equation}
\mbby_{N}(\theta) :=\frac{1}{N}\left(\mbbh_{N}(\theta) - \mbbh_{N}(\tz)\right).
\nn
\end{equation}
Let us write $\mbby_{N}(\theta)=N^{-1}\sumi \xi_i(\theta)$, where
\begin{align}
\xi_i(\theta) &:= \frac{1}{2} \Big(\log|\Sig_i(v_0)| - \log|\Sig_i(v)| - \Sig_i(v)^{-1} [(\mu_i(\beta_0) - \mu_i(\beta))^{\otimes2}] \nn \\
&{} \qquad + (\Sig_i(v_0)^{-1} - \Sig_i(v)^{-1})[(Y_i - \mu_i(\beta_0))^{\otimes2}] 
\nn\\
&{}\qquad - 2\Sig_i(v)^{-1}[\mu_i(\beta_0) - \mu_i(\beta), Y_i - \mu_i(\beta_0)] \Big).
\nn
\end{align}
For each $\theta$,
\begin{align}
    \E[\xi_i(\theta)] &= \frac{1}{2} \Big(
    \log|\Sig_i(v_0)| - \log|\Sig_i(v)| 
    - \tr\left(\Sig_i(v)^{-1} \Sig_i(v_0) - I_{n_i}\right)
    \nn\\
    &\qquad - \Sig_i(v)^{-1} [(\mu_i(\beta) - \mu_i(\beta_0))^{\otimes2}]\Big). \nn
\end{align}
Let
\begin{align}
F_{N,1}(v)&:=\frac{1}{N} \sumi \{\log|\Sigma_i(v_0)| -\log|\Sigma_i(v)| - \tr\left(\Sigma_i(v)^{-1} \Sigma_i(v_0) - I_{n_i}\right)\},
\nn\\
F_{N,2}(\theta)&:= \frac{1}{N}\sumi \Sigma_i(v)^{-1} [(\mu_i(\beta) - \mu_i(\beta_0))^{\otimes2}].
\nn
\end{align}
Note that we are not making any structural assumptions on the sequences $(X_{ij})_{j=1}^{n_i}$ and $(Z_{ij})_{j=1}^{n_i}$ for all $i=1, \dots, N$. To ensure the convergence of $\mbby_N(\cdot)$ to a specific limit in probability, we impose the following.

\begin{ass}\label{hm:A_3}
There exist non-random $\mcc^{2}(\Theta)$-functions $F_{0,1}(v)=F_{0,1}(v;v_0)$ and $F_{0,2}(\theta)=F_{0,2}(\theta;\beta_0)$ such that 
%\tcm{[$c_k \gets 1/2$ is necessary to make the 2nd-order coefficients $O_p(1)$]}
\begin{align}
\sup_N \sup_\theta \left( \sqrt{N} |F_{N,1}(v) - F_{0,1}(v)| + \sqrt{N} |F_{N,2}(\theta) - F_{0,2}(\theta)| \right) <\infty.
\nn
\end{align}
and that $F_{0,1}(v)$ and $F_{0,2}(\theta)$ and their partial derivatives of orders $\le 2$ are continuous in $\overline{\Theta}$.
\end{ass}

\begin{rem}
\label{hm:rem_LLN-with-rate}
    In our setting, the explicit forms of $F_{0,1}(v)$ and $F_{0,2}(\theta)$ are not available in general because of the possible unbalanced nature of the longitudinal data under consideration; unfortunately, it is the case even when we assume that $(X_i, Z_i)$ is the sequence of i.i.d. random processes. Concerned with the identification of the limits, we have the same situations in Assumptions \ref{hm:A-S.lim} and \ref{hm:A-Gam} below; still, the situation could be simplified to some extent when $\mu_i(\beta)=X_i \beta$ (see Section \ref{hm:sec_linear.case}).
    \rrev{Importantly, as we will demonstrate in Corollary \ref{hm:cor_jASN}, the construction of an asymptotic confidence set based on our estimator is straightforward.}
\end{rem}

Assumption \ref{hm:A_3} implies that
\begin{align}\label{hm:lc-1}
\sup_N \sup_{\theta \in \Theta} \left|\sqrt{N} \left(\frac{1}{N} \sumi \E[\xi_i(\theta)] - \mbby_0(\theta)\right)\right| < \infty,
\end{align}
where
\begin{equation}
    \mbby_0(\theta) := \frac{1}{2} \left( F_{0,1}(v) - F_{0,2}(\theta)\right)
\end{equation}
is a non-random $\mcc^{2}(\Theta)$-function.
We see that $F_{0,1}(v)\le 0$ by invoking the property of the Kullback-Leibler divergence between two multivariate normal distributions. Since $F_{0,2}(\theta)\ge 0$, it holds that $\mbby_0(\theta)\le 0$. 
We follow the custom of \cite{Yos11} to state the identifiability condition:

\begin{ass}\label{hm:A_4}
There exists a constant $\chi_0>0$ such that 
% $\mbby_0(\theta) \leq - \chi_0 |\theta - \tz|^2$ 
$F_{0,1}(v)\le -\chi_0|v-v_0|^2$ and $F_{0,2}(\theta) \ge \chi_0|\theta-\tz|^2$ for every $\theta\in\Theta$.
\end{ass}

The following two conditions are sufficient for Assumption \ref{hm:A_4}:
\begin{itemize}
    \item $\{\tz\} = \argmax_{\theta \in \overline{\Theta}} \mbby_0(\theta)$, namely, $\mbby_0(\theta)=0$ if and only if $\theta=\tz$;
    \item $-\p_{\theta}^2 \mbby_0(\tz)$ is positive definite.
\end{itemize}
The sufficiency can be seen through the Taylor expansion and the compactness of $\overline{\Theta}$: 
first, take $\del>0$ small enough to ensure that
% if $|\theta-\tz|\le \del$,
\begin{equation}\nn
    \sup_{\theta:\,|\theta-\tz|\le \del}
    |\theta-\tz|^{-2}\mbby_0(\theta) \lesssim - \inf_{\widetilde{\theta}:\,|\widetilde{\theta}-\tz|\le \del}\lam_{\min}\big(-\p_\theta^2 \mbby_0(\widetilde{\theta})\big) \le -\chi_{0,1}
\end{equation}
for some $\chi_{0,1}>0$;
second, with the so-chosen $\del>0$, 
% for each $\theta$ such that $|\theta-\tz|\ge \del$, 
the compactness of $\overline{\Theta}$ implies that
\begin{equation}\nn
    \sup_{\theta:\,|\theta-\tz|> \del}|\theta-\tz|^{-2}\mbby_0(\theta) \le -\chi_{0,2}
\end{equation}
for some $\chi_{0,2}>0$.
Hence Assumption \ref{hm:A_4} is verified with $\chi_0=\min\{\chi_{0,1},\chi_{0,2}\}$.

% By the standard $M$-estimation theory, 
Under Assumption \ref{hm:A_4}, the consistency $\tes\cip\tz$ follows from the uniform convergences in probability $\sup_\theta|\mbby_N(\theta) - \mbby_0(\theta)|\cip 0$.
We will derive it in the following stronger form:
\begin{align}\label{hm:lc-2}
% \exists c_2 \in (0, 1/2]~
\forall K > 0,\quad
\sup_N \E\left[\sup_{\theta} \left(\sqrt{N} |\mbby_N(\theta) - \mbby_0(\theta)|\right)^K\right] < \infty.
\end{align}
Observe that
\begin{align}
\mbby_N(\theta) - \mbby_0(\theta) = \frac{1}{N}\sumi \left(\xi_i(\theta) - \E[\xi_i(\theta)]\right) + \frac{1}{N} \sumi \E[\xi_i(\theta)] - \mbby_0(\theta).\nn
\end{align}
By \eqref{hm:lc-1}, for \eqref{hm:lc-2} it remains to look at the first term on the right-hand side.
We will make use of the following basic uniform moment estimates. 
%of the martingale term. 
Recall that $p$ denotes the dimension of $\theta$.

\begin{lem}
\label{hm:lem_moment.estimate}
Let $\Theta\subset\mbbr^p$ be a bounded convex domain, $q>p\vee 2$, and let $\chi_{Ni}(\theta): \Theta\to\mbbr$, $i\le N$, $N\ge 1$, be random functions. Then, we have
\begin{align}
    \E\left[\sup_\theta \left|\sumi \chi_{Ni}(\theta)\right|^q\right]
\lesssim 
\sup_\theta \E\left[\left|\sumi \chi_{Ni}(\theta)\right|^q\right]
+
\sup_\theta \E\left[\left|\sumi \p_\theta \chi_{Ni}(\theta)\right|^q\right]
    \nn
\end{align}
If in particular $(\p_\theta^k \chi_{Ni}(\theta))_{i=1}^N$ for $k\in\{0,1\}$ and $\theta\in\Theta$ forms a martingale difference array with respect to some filtration $(\mcf_{Ni})_{i\le N}$, then
\begin{align}
\E\left[\sup_\theta \left|\frac{1}{\sqrt{N}}\sumi \chi_i(\theta)\right|^q\right]
\lesssim 
\sup_\theta \frac1N \sumi \E\left[|\chi_i(\theta)|^q\right]
+
\sup_\theta \frac1N \sumi \E\left[|\p_\theta\chi_i(\theta)|^q\right].
% &\lesssim 
% \sup_\theta \E\left[\left|\frac{1}{\sqrt{N}}\sumi \chi_i(\theta)\right|^q\right]
% +
% \sup_\theta \E\left[\left|\frac{1}{\sqrt{N}}\sumi \p_\theta \chi_i(\theta)\right|^q\right].
\nn
\end{align}
\end{lem}

\begin{proof}
The first inequality is due to the Sobolev inequality \cite{AdaFou03} which says that $\sup_\theta |f(\theta)| \lesssim \int_{\Theta}(|f(\theta)|+|\p_\theta f(\theta)|)d\theta$. Then, we can apply the Burkholder inequality to obtain the second one.
% Lemma \ref{hm:lem_moment.estimate} has several trivial extensions.
\end{proof}

Returning to our model setup, by Lemma \ref{hm:lem_moment.estimate} we have
\begin{align}
& \hspace{-2cm}
\E\left[\sup_{\theta \in \Theta} \left|\frac{1}{\sqrt{N}} \sumi \left(\xi_i(\theta) - \E[\xi_i(\theta)]\right)\right|^K\right] \nn\\
&\lesssim \sup_{\theta \in \Theta} \E\left[\left|\frac{1}{\sqrt{N}}\sumi \left(\xi_i(\theta) - \E[\xi_i(\theta)]\right)\right|^K\right] \nn \\
&\qquad + \sup_{\theta \in \Theta} \E\left[\left|\frac{1}{\sqrt{N}} \sumi \left(\p_{\theta} \xi_i(\theta) - \E[\p_{\theta} \xi_i(\theta)]\right)\right|^K\right].
\nn
\end{align}
From Burkholder's inequality and Jensen's inequality, it follows that
\begin{align}
\E\left[\left|\frac{1}{\sqrt{N}} \sumi \left(\xi_i(\theta) - \E[\xi_i(\theta)]\right)\right|^K\right] &\lesssim \E\left[\left(\frac1N \sumi \left|\xi_i(\theta) - \E[\xi_i(\theta)]\right|^2\right)^{K/2}\right] \nn \\
&\leq \frac{1}{N} \sumi \E\left[\left|\xi_i(\theta) - \E[\xi_i(\theta)]\right|^K\right]. \nn
\end{align}
For every $K>0$, we have
\begin{equation}\label{hm:Y-bm}
    \sup_{N\ge 1}\max_{1 \leq i \leq N}\E[|Y_i|^K] < \infty
\end{equation}
hence also $\sup_{N\ge 1}\max_{1 \leq i \leq N}\sup_\theta \E[|\xi_i(\theta)|^K] < \infty$.
Therefore,
\begin{align}
\sup_N \sup_{\theta \in \Theta} \E\left[\left|\frac{1}{\sqrt{N}} \sumi \left(\xi_i(\theta) - \E[\xi_i(\theta)]\right)\right|^K\right] < \infty. \nn
\end{align}
Similarly, we obtain
\begin{align}
\sup_N \sup_{\theta \in \Theta} \E\left[\left|\frac{1}{\sqrt{N}} \sumi \left(\p_{\theta}\xi_i(\theta) - \E[\p_{\theta}\xi_i(\theta)]\right)\right|^K\right] < \infty.
\nn
\end{align}
From these we obtain \eqref{hm:lc-2}, and hence the consistency $\tes\cip\tz$ holds.

%%%
\subsection{Quasi-score function}
Define the quasi-score function by
\begin{align}\nn
\D_N(\theta) 
% = 
% \begin{pmatrix}
% \D_{N, \beta}(\theta) \\
% \D_{N, v}(\theta)
% \end{pmatrix}
= \frac{1}{\sqrt{N}}\p_{\theta}\mbbh_N(\theta).
% =: \frac{1}{\sqrt{N}} \sumi \D_{N, i}(\theta) ,
\end{align}
We have $\D_N(\theta) = (\D_{N, \beta}(\theta),\,\D_{N, v}(\theta)) \in \mbbr^{p_\beta}\times\mbbr^{p_v}$ with
\begin{align}
\D_{N, \beta}(\theta)
&:= \frac{1}{\sqrt{N}} \sumi \Sig_i(v)^{-1}[\p_{\beta} \mu_i(\beta), Y_i - \mu_i(\beta)], \label{hm:def_D_beta}\\
\D_{N, v}(\theta)
&:= \frac{1}{\sqrt{N}} \sumi 
\bigg(\frac{1}{2} (\Sig_i(v)^{-1} (\p_{v_j} \Sig_i(v)) \Sig_i(v)^{-1})[(Y_i - \mu_i(\beta))^{\otimes2}] 
\nn\\
&{}\qquad - \frac{1}{2} \tr(\Sig_i(v)^{-1} (\p_{v_j}\Sig_i(v)))\bigg)_{j=1}^{p_v}.
\label{hm:def_D_v}
\end{align}
From now on, we will often omit ``$(\tz)$'', ``$(\beta_0)$'', and ``$(v_0)$'' from the notation, such as $\D_N=\D_N(\tz)$.
% We will write $\D_N$ for $\D_N(\tz) = (\D_{N, \beta}(\tz),\,\D_{N, v}(\tz))$.
Obviously, $\E[\D_{N, \beta}]=\E[\D_{N, v}]=0$.
Let
\begin{equation}
    A_{ij}%=A_{ij}(v_0) 
    := \Sig_i^{-1} (\p_{v_j}\Sig_i) \Sig_i^{-1}.
\end{equation}
% with $A_{ij}(v_0) := \Sig_i(v_0)^{-1} (\p_{v_j}\Sig_i(v_0)) \Sig_i(v_0)^{-1}$,
Then, the covariance matrix
\begin{align} 
S_N := \cov[\D_N] &= 
\begin{pmatrix}
\cov[\D_{N, \beta}] & \cov[\D_{N, \beta},\D_{N, v}] \\
\cov[\D_{N, \beta},\D_{N, v}]^\top & \cov[\D_{N, v}]
\end{pmatrix} \nn \\
&=:
\begin{pmatrix}
S_{N,11} & S_{N,12} \\
S_{N,12}^\top & S_{N,22}
\end{pmatrix} \label{ti:CovScore}
\end{align}
is given by
\begin{align}
S_{N,11} &= \frac{1}{N} \sumi \Sig_i^{-1} [(\p_{\beta} \mu_i)^{\otimes2}]
\in \mbbr^{p_\beta}\otimes\mbbr^{p_\beta}, \nn \\
S_{N,12} &= \frac{1}{2N} \sumi \left(\Sig_i^{-1} [\p_{\beta} \mu_i, \E[(Y_i - \mu_i )^{\otimes2}A_{ij} (Y_i - \mu_i)]]\right)_{j=1}^{p_v}
\in\mbbr^{p_\beta} \otimes \mbbr^{p_v}, \nn \\
S_{N,22} &= \frac{1}{4N} \sumi \bigg(\E[\tr(A_{ij}(Y_i - \mu_i^{\otimes2})\cdot \tr(A_{ik}(Y_i - \mu_i)^{\otimes2})] \nn \\
    & {}\qquad - \tr\left(\Sig_i^{-1} \p_{v_j}\Sig_i\right)\cdot \tr\left(\Sig_i^{-1} \p_{v_k}\Sig_i\right)\bigg)_{j, k=1}^{p_v} \in \mbbr^{p_v}\otimes\mbbr^{p_v}.
\nn
\end{align}
To identify the asymptotic covariance of $\D_N$, we need the convergence of \eqref{ti:CovScore}.

\begin{ass}\label{hm:A-S.lim}
There exists a positive definite matrix 
\begin{equation}\nn
S_0 =
\begin{pmatrix}
S_{0,11} & S_{0,12} \\
S_{0,12}^\top & S_{0,22}
\end{pmatrix}
\in \mbbr^{p}\otimes\mbbr^{p}    
\end{equation}
such that $(S_{N,11},S_{N,12},S_{N,22}) \rightarrow (S_{0,11},S_{0,12},S_{0,22})$, hence $S_N \rightarrow S_0$, as $N \rightarrow \infty$.
\end{ass}

Let $\D_N=:\sumi N^{-1/2}\psi_{i}$. We have $\E[\psi_{i}]=0$ as was mentioned, hence $\sumi\cov[N^{-1/2}\psi_{i}]=\sumi\E[(N^{-1/2}\psi_{i})^{\otimes 2}] = S_N \to S_0$.
% It is easily seen that for $\del>0$,
% \begin{align}
% & \sumi \E[|N^{-1/2}\psi_i|^{2+\del}] \nn\\
% &\lesssim N^{-\del/2} \cdot \frac{1}{N} \sumi \left\{|\p_{\beta} \mu_i(\beta_0)|^{2+\del} \cdot |\Sig_i(v_0)^{-1}|^{2+\del} \cdot \E[|Y_i - \mu_i(\beta_0)|^{2+\del}] \right. \nn \\
% &\left. \quad + \left(\frac{1}{2}\right)^{2+\del} \sum_{j=1}^{p_v} \left\{|A_{ij}|^{2+\del} \E[|Y_i - \mu_i(\beta_0)|^{2(2+\del)}] + |\tr\left(\Sig_i(v_0)^{-1} \p_{v_j}\Sig_i(v_0)\right)|^{2+\del}\right\}\right\} \nn \\
% & \rightarrow 0.
% \nn
% \end{align}
Trivially, for each $\del>0$, 
% $\sumi\E[|N^{-1/2}\psi_{i}|^{2+\del}] =O(N^{-\del/2}) \to 0$
\begin{equation}\nn
    \sumi\E[|N^{-1/2}\psi_{i}|^{2+\del}] =O(N^{-\del/2}) \to 0,
\end{equation}
since $\max_{i\le N}\E[|\psi_{i}|^{2+\del}]=O(1)$, so that the Lyapunov condition holds.
Accordingly, the Lindeberg-Lyapunov central limit theorem concludes that
\begin{equation}\label{hm:D-clt}
    \D_N \cil N_p(0,S_0).
\end{equation}
Further, by Burkholder's inequality and Jensen's inequality,
\begin{align}
\E\left[|\D_N|^{K}\right] = \E\left[\left|\sumi \frac{1}{\sqrt{N}}\psi_i\right|^K \right] \lesssim \E\left[\left(\frac{1}{N}\sumi |\psi_i|^2\right)^{K/2}\right] \leq \frac{1}{N} \sumi \E[|\psi_i|^K], \nn
\end{align}
so that
\begin{align}\label{hm:s.oi-1}
\forall K \ge 2,\quad \sup_N \E\left[|\D_N|^K\right] < \infty.
\end{align}

%%%
\subsection{Quasi-observed information}
Define the quasi-observed information matrix by
\begin{align}
\Gam_N(\theta) &:= -\frac{1}{N}\p_{\theta}^{2}\mbbh_{N}(\theta) = 
\begin{pmatrix}
\Gam_{N,11}(\theta) & \Gam_{N,12}(\theta) \\
\Gam_{N,12}(\theta)^\top & \Gam_{N,22}(\theta)
\end{pmatrix},
\label{ti:CovObsInf}
\end{align}
where
\begin{align}
\Gam_{N,11}(\theta) &:= \frac{1}{N} \sumi \{\Sig_i(v)^{-1} [(\p_{\beta} \mu_i(\beta))^{\otimes2}] - \Sig_i(v)^{-1}[\p_{\beta}^2\mu_i(\beta), Y_i - \mu_i(\beta)]\}, \nn \\
\Gam_{N,12}(\theta) &:= \frac{1}{N} \sumi \left((\Sig_i(v)^{-1} (\p_{v_j} \Sig_i(v)) \Sig_i(v)^{-1}) [\p_{\beta} \mu_i(\beta), Y_i - \mu_i(\beta)]\right)_{j = 1}^{p_v}, \nn \\
\Gam_{N,22}(\theta) &:= \frac{1}{N} \sumi \bigg((\Sig_i(v)^{-1} (\p_{v_j} \Sig_i(v)) \Sig_i(v)^{-1} (\p_{v_k} \Sig_i(v)) \Sig_i(v)^{-1}) [(Y_i - \mu_i(\beta))^{\otimes2}] \nn \\
&\left. \qquad -\frac{1}{2} (\Sig_i(v)^{-1} (\p_{v_j v_k}^2 \Sig_i(v)) \Sig_i(v)^{-1}) [(Y_i - \mu_i(\beta))^{\otimes2}] \right. \nn \\
&\qquad + \frac{1}{2} \tr\big(-\Sig_i(v)^{-1} (\p_{v_j} \Sig_i(v)) \Sig_i(v)^{-1} (\p_{v_k} \Sig_i(v)) \nn\\
&{}\qquad\qquad + \Sig_i(v)^{-1} (\p_{v_j v_k}^2 \Sig_i(v))\big)\bigg)_{j, k = 1}^{p_v}
\nn 
\end{align}
of sizes $p_\beta\times p_\beta$, $p_\beta\times p_v$, and $p_v\times p_v$, respectively.

As in Assumptions \ref{hm:A_3} and \ref{hm:A-S.lim}, we need the following for the asymptotic behavior of the non-random sequence $\Gam_N=\Gam_N(\tz)$ in \eqref{ti:CovObsInf}.

\begin{ass}\label{hm:A-Gam}
There exists a block-diagonal matrix
\begin{equation}\nn
    \Gam_0 = \mathrm{diag}(\Gam_{0,11},\Gam_{0,22})\in \mbbr^{p}\otimes\mbbr^{p}
\end{equation}
with both $\Gam_{0,11}\in \mbbr^{p_\beta}\otimes\mbbr^{p_\beta}$ and $\Gam_{0,22}\in \mbbr^{p_v}\otimes\mbbr^{p_v}$ being positive definite such that
\begin{equation}
    \sup_N \sqrt{N}\left|\frac{1}{N}\sumi\Sig_i^{-1} [(\p_{\beta} \mu_i)^{\otimes2}]
    - \Gam_{0,11}\right| < \infty
    \nn
\end{equation}
and that
\begin{equation}
    \sup_N \sqrt{N} \left|\frac{1}{N}\sumi
    \left(\frac{1}{2} \tr(\Sig_i^{-1} (\p_{v_j} \Sig_i) \Sig_i^{-1} (\p_{v_k} \Sig_i))\right)_{j, k = 1}^{p_v} - \Gam_{0,22}\right| < \infty.
    \nn
\end{equation}
\end{ass}

\medskip

Fix any $K \geq 2$ and write $\Gam_N=N^{-1}\sumi \Gam_{N,i}$.
Then, $\sup_{i\le N}\E[|\Gam_{N,i}|^K]=O(1)$.
% We have
% \begin{equation}
% \Gam_{N, i} :=
% \begin{pmatrix}
% \Gam_{N, i, 1} & \Gam_{N, i, 2} \\
% \Gam_{N, i, 2}^\top & \Gam_{N, i, 3}
% \end{pmatrix}
% \nn
% \end{equation}
% with
% \begin{align}
% \Gam_{N, i, 1} &= \Sig_i(v_0)^{-1} [(\p_{\beta} \mu_i(\beta_0))^{\otimes2}] - \Sig_i(v_0)^{-1}[\p_{\beta}^2\mu_i(\beta_0), Y_i-\mu_i(\beta_0)], \nn \\
% \Gam_{N, i, 2} &= \left((\Sig_i(v_0)^{-1} (\p_{v_j} \Sig_i(v_0)) \Sig_i(v_0)^{-1}) [\p_{\beta} \mu_i(\beta_0), Y_i - \mu_i(\beta_0)]\right)_{j = 1}^{p_v}, \nn \\
% \Gam_{N, i, 3} &= \left((\Sig_i(v_0)^{-1} (\p_{v_j} \Sig_i(v_0)) \Sig_i(v_0)^{-1} (\p_{v_k} \Sig_i(v_0)) \Sig_i(v_0)^{-1}) [(Y_i - \mu_i(\beta_0))^{\otimes2}] \right. \nn \\
% &\left. \qquad -\frac{1}{2} (\Sig_i(v_0)^{-1} (\p_{v_j v_k}^2 \Sig_i(v_0)) \Sig_i(v_0)^{-1}) [(Y_i - \mu_i(\beta_0))^{\otimes2}] \right. \nn \\
% &\left. \qquad + \frac{1}{2} \tr\left(-\Sig_i(v_0)^{-1} (\p_{v_j} \Sig_i(v_0)) \Sig_i(v_0)^{-1} (\p_{v_k} \Sig_i(v_0)) + \Sig_i(v_0)^{-1} \p_{v_j v_k}^2 \Sig_i(v_0)\right)\right)_{j, k = 1}^{p_v}. \nn 
% \end{align}
It follows from Burkholder's inequality and Jensen's inequalities that 
\begin{align}
\sup_N\E\left[\left(\sqrt{N}|\Gam_N - \E[\Gam_N]|\right)^K\right] 
&\lesssim 
% \sup_N\E\left[N^{K/2}\left(\frac{1}{N^2} \sumi |\Gam_{N, i} - \E[\Gam_{N, i}]|^2\right)^{K/2}\right] \nn \\
% &= 
\sup_N\E\left[\left(\frac{1}{N} \sumi |\Gam_{N, i} - \E[\Gam_{N, i}]|^2\right)^{K/2}\right] \nn \\
& \leq \sup_N\frac{1}{N} \sumi \E\left[|\Gam_{N, i} - \E[\Gam_{N, i}]|^K\right] <\infty.
\label{hm:qoi-3}
\end{align}
We have $\E[\Gam_{N,12}]=0$ and it is easy to see that $\sup_N N^{1/2} |\E[\Gam_N] - \Gam_0| < \infty$ under Assumption \ref{hm:A-Gam}.
This combined with \eqref{hm:qoi-3} concludes
\begin{align}\label{hm:qoi-2}
\sup_N \E\left[\left(N^{1/2}|\Gam_N - \Gam_0|\right)^K\right] < \infty,
\end{align}
in particular $\Gam_N \cip \Gam_0$.

%%%
\subsection{Asymptotic normality and tail-probability estimate}
\label{hm:sec_an-tpe}
Let
% $\hat{u}_{N}:=\sqrt{N} (\tes - \tz)=(\sqrt{N} (\bes - \beta_0),\sqrt{N} (\hat{v}_N - v_0))$. 
\begin{equation}\nn
    \hat{u}_{N}:=\sqrt{N} (\tes - \tz)=\left(\sqrt{N} (\bes - \beta_0),\sqrt{N} (\hat{v}_N - v_0)\right).
\end{equation}
The following theorem is the main claim of this section.

\begin{thm} \label{ti:QLA.thm1}
Let Assumptions \ref{hm:A_1}, \ref{hm:A_2}, \ref{hm:A_3}, \ref{hm:A_4}, \ref{hm:A-S.lim}, and \ref{hm:A-Gam} hold.
\begin{enumerate}
\item \label{ti:QLA.thm1-1}
We have the stochastic expansion
\begin{align}\label{ti:QLA.thm1-(1)}
% \sqrt{N} (\tes - \tz)
\hat{u}_{N}= G_{N,1} + \frac{1}{\sqrt{N}} G_{N,2} + O_p\left(\frac1N\right),
\end{align}
where
%\tcm{[
%For the 2nd-order SE, we need $c_1=1/2$ to ensure $\sqrt{N}(\Gam_0 - \Gam_N)=O_p(1)$ so that $G_{1,N}=O_p(1)$.
%]}
\begin{align}
G_{N,1} &= \Gam_0^{-1} \D_N, \label{hm:def_G_1N}\\
G_{N,2} &= \Gam_0^{-1} \bigg\{(\sqrt{N}(\Gam_0 - \Gam_N)) [\Gam_0^{-1}\D_N] 
\nn\\
&{}\qquad 
+ \frac{1}{2} \left(\frac{1}{N} \p_\theta^3 \mbbh_N(\tz)\right)[(\Gam_0^{-1} \D_N)^{\otimes 2}]\bigg\}.
\label{hm:def_G_2N}
\end{align}
% where $G_{1, N}$ is sufficiently integrable.
In particular,
\begin{equation}\label{hm:u-hat-AN}
\hat{u}_{N} \cil N_p\left(0,\, \Gam_0^{-1} S_0 \Gam_0^{-1}\right).
\end{equation}

\item \label{ti:QLA.thm1-2}
For any $L > 0$ there exists a universal constant $C_L > 0$ for which
\begin{align}\label{ti:QLA.thm1-(2)}
\sup_N P[|\hat{u}_N| > r] \leq \frac{C_L}{r^L}, \qquad r > 0.
\end{align}

\end{enumerate}
\end{thm}

It immediately follows from \eqref{ti:QLA.thm1-(2)} that the random sequence $(\hat{u}_N)_N$ is $L^q$-bounded for any $q>0$, hence the convergence of moments $\E\left[f(\hat{u}_N)\right] \to \E\left[f(\hat{u}_0)\right]$, where $\hat{u}_0\sim N_{p}\left(0,\Gam_0^{-1}S_0\Gam_0^{-1}\right)$, holds for any continuous function $f$ of at most polynomial growth.
We note that $S_{12,0}=O$ if the distributions of $b_i$, $L_i(1)$, $\ep_i$ are all symmetric, so that
$\hat{u}_{N} \cil N_p\left(0,\, \diag\{\Gam_{0,11}^{-1} S_{0,11} \Gam_{0,11}^{-1},\, \Gam_{0,22}^{-1} S_{0,22} \Gam_{0,22}^{-1}\}\right)$.

\rrev{The proof procedure for Theorem \ref{ti:QLA.thm1} is outlined below. 
\begin{itemize}
    \item The stochastic expansion of $\hat{u}_N$ primarily based on the Taylor expansion of the quasi-score function. To evaluate $G_{N,2}$ and to show that the third- or higher-order terms in the expansion of $\hat{u}_N$ are $O_p(1/N)$, we introduce Lemma \ref{ti:boundary of three and forth dev} below.   
    \item The asymptotic normality of $\hat{u}_N$ is derived by combining its stochastic expansion with the asymptotic distribution of the quasi-score function \eqref{hm:D-clt}.
    \item The tail-probability estimate of $\hat{u}_N$ is deduced by applying \cite[Theorem 3]{Yos11}.
\end{itemize}
While the evaluation of $G_{N,2}$ is not strictly required for proving asymptotic normality of $\hat{u}_N$, it is indispensable for showing the difference from the stochastic expansion of the stepwise GQMLE discussed in the following section.
}

\begin{lem}
\label{ti:boundary of three and forth dev}
\begin{align}
\forall K > 0,\quad 
\sup_N \E \left[\sup_{\theta} \left(\left|\frac{1}{N} \p_{\theta}^3 \mbbh_N(\theta)\right|^K + \left|\frac{1}{N} \p_{\theta}^4 \mbbh_N(\theta)\right|^K\right)\right] < \infty.
\nn
\end{align}
\end{lem}

\begin{proof}
The components of the third-order derivative $\p_{\theta}^3 \mbbh_N(\theta)$ are explicitly given as follows:
 \begin{align}
\p_{\beta}^3 \mbbh_N(\theta) &= \sumi \{\Sig_i(v)^{-1}[\p_{\beta}^3\mu_i(\beta), Y_i-\mu_i(\beta)] - 3\Sig_i(v)^{-1}[\p_{\beta}^2\mu_i(\beta), \p_{\beta}\mu_i(\beta)]\} \nn \\
 \p_{\beta}^2 \p_{v} \mbbh_N(\theta) &= \sumi \left(\p_{v_j} \Sig_i(v)^{-1} [(\p_{\beta} \mu_i(\beta))^{\otimes2}]\right)_{j=1}^{p_v}, \nn \\
 \p_{\beta} \p_{v}^2 \mbbh_N(\theta) &= \sumi \left(\p_{v_j v_k}^2 \Sig_i(v)^{-1} [\p_{\beta} \mu_i(\beta), Y_i - \mu_i(\beta)]\right)_{j, k=1}^{p_v}, \nn \\
 \p_{v}^3 \mbbh_N(\theta) &= \sumi \bigg(-\frac{1}{2} \p_{v_j v_k v_l}^3 \Sig_i(v)^{-1} [(Y_i - \mu_i(\beta))^{\otimes2}] \nn \\
&{} \qquad + \frac{1}{2} \tr\left\{-\p_{v_l}\left(\Sig_i(v)^{-1}(\p_{v_j} \Sig_i(v)) \Sig_i(v)^{-1} (\p_{v_k} \Sig_i(v)^{-1})\right)\right\}
\nn\\
&{}\qquad
+ \p_{v_l}\left(\Sig_i(v)^{-1} \p_{v_j v_k}^2\Sig_i(v)\right)\bigg
)_{j, k, l=1}^{p_v}.
\nn
\end{align}
Recalling \eqref{hm:Y-bm}, it is easy to see that $\sup_N \E [\sup_{\theta}| N^{-1} \p_{\theta}^3 \mbbh_N(\theta)|^K] < \infty
$.
The case of the fourth-order derivative $\p_{\theta}^4 \mbbh_N(\theta)$ is similar, hence omitted.
\end{proof}

\begin{proof}[Proof of Theorem \ref{ti:QLA.thm1}]\eqref{ti:QLA.thm1-1} By the Taylor expansion of $\D_N(\tes)$ around $\tz$, 
\begin{equation}
\D_N(\tes) = \D_N + \frac{1}{N} \p_{\theta}^2 \mbbh_N(\tz)[\hat{u}_N] + \frac{1}{2\sqrt{N}} \frac{1}{N} \p_{\theta}^3\mbbh_N(\widetilde{\theta}_N)[\hat{u}_N^{\otimes2}], \nn
\end{equation}
where $|\widetilde{\theta}_N-\tz|\leq|\tes-\tz|$.
By the consistency, we may and do set $\D_N(\tes)=0$; similar remarks apply to the stepwise version in Section \ref{hm:sec_stepwise}. Then,
\begin{align}
\Gam_0[\hat{u}_N]
&= \D_N + \frac{1}{\sqrt{N}}\left(\frac{1}{\sqrt{N}}(\p_{\theta}^2\mbbh_N(\tz)+N\Gam_0)[\hat{u}_N]\right)
+ \frac{1}{2\sqrt{N}} \left(\frac{1}{N} \p_{\theta}^3\mbbh_N(\widetilde{\theta}_N)[\hat{u}_N^{\otimes2}]\right) \nn \\
&= \D_N + \frac{1}{\sqrt{N}}\bigg((-\sqrt{N}(\Gam_N - \Gam_0))[\hat{u}_N] 
+ \frac{1}{2N} \p_{\theta}^3 \mbbh_N(\widetilde{\theta}_N)[\hat{u}_N^{\otimes 2}]\bigg).
\label{ti:2nd.se1}
\end{align}
It follows that
\begin{equation}
\hat{u}_N = \Gam_0^{-1} \D_N + O_p(N^{-1/2}).
\label{ti:1st.se}
\end{equation}
By \eqref{hm:D-clt} and \eqref{hm:qoi-2}, we get $(\D_N, \Gam_N) \cil (S_0^{1/2} \eta, \Gam_0)$, where $\eta \sim N_p(0, I_p)$.
Hence \eqref{ti:1st.se} gives \eqref{hm:u-hat-AN}.
Substituting $(\ref{ti:1st.se})$ for the right-hand side of $(\ref{ti:2nd.se1})$, we get
\begin{align}
\Gam_0[\hat{u}_N] &= \D_N + \frac{1}{\sqrt{N}}\bigg\{(-\sqrt{N}(\Gam_N - \Gam_0))[\Gam_0^{-1}\D_N] \nn\\
&{}\qquad + \frac{1}{2N} \p_{\theta}^3 \mbbh_N(\widetilde{\theta}_N)[(\Gam_0^{-1}\D_N)^{\otimes 2}]\bigg\} + O_p(N^{-1}).
\label{ti:2nd.se2}
\end{align}
By Lemma \ref{ti:boundary of three and forth dev}, we have $N^{-1}\p_{\theta}^3 \mbbh_N(\widetilde{\theta}_N) = N^{-1}\p_{\theta}^3 \mbbh_N(\tz) + O_p(N^{-1/2})$.
Therefore,
\begin{align}
\hat{u}_N &= \Gam_0^{-1}\D_N+ \frac{1}{\sqrt{N}} \Gam_0^{-1} \bigg\{(-\sqrt{N}(\Gam_N - \Gam_0)) [\Gam_0^{-1}\D_N] 
\nn\\
&{}\qquad + \frac{1}{2} \left(\frac{1}{N}\p_{\theta}^3\mbbh_N(\tz)\right)[(\Gam_0^{-1} \D_N)^{\otimes 2}]\bigg\} + O_p(N^{-1}).
\nn
\end{align}
This completes the proof of (1).%\eqref{ti:QLA.thm1-1}.

\eqref{ti:QLA.thm1-2}
Based on the estimates \eqref{hm:lc-2}, \eqref{hm:s.oi-1}, \eqref{hm:qoi-2}, and Lemma \ref{ti:boundary of three and forth dev}, the claim %\eqref{ti:QLA.thm1-2} 
readily follows from the general machinery of \cite[Theorem 3]{Yos11}.
\end{proof}

\rev{
\begin{rem}
The asymptotic covariance of $\widehat{u}_N$ given by \eqref{hm:u-hat-AN} is composed of third and fourth moments of the random-effect, the system noise, and the measurement error from \eqref{ti:CovScore}, \eqref{ti:CovObsInf}. Therefore, as the higher-order moments of these random elements become larger, the standard error of the joint GQMLE is likely to become larger. 
The same holds for the asymptotic covariance of the stepwise GQMLE, given by \eqref{hm:stepwise.se} in the same form as in \eqref{hm:u-hat-AN}.
% This also holds for the stepwise GQMLE, which has the same asymptotic covariance. The asymptotic covariance of the stepwise GQMLE will be given by \eqref{hm:stepwise.se}.
\end{rem}
}

We now discuss how to construct an approximate confidence set.
Let
\begin{equation}\nn
    \hat{A}_{i,j} := \Sig_i(\hat{v}_N)^{-1}(\p_{v_j}\Sig_i(\hat{v}_N)) \Sig_i(\hat{v}_N)^{-1}.
\end{equation}
Let
\begin{equation}\nn
\hat{S}_N := 
\begin{pmatrix}
\hat{S}_{N,11} & \hat{S}_{N,12} \\
\hat{S}_{N,12}^\top & \hat{S}_{N,22}
\end{pmatrix},
\qquad 
\hat{\Gam}_{N}=\diag(\hat{\Gam}_{N,11},\hat{\Gam}_{N,22}),
\end{equation}
where
\begin{align}
\hat{S}_{N,11} &= \frac{1}{N} \sumi \Sig_i^{-1}(\hat{v}_N) [(\p_{\beta}\mu_i(\hat{\beta}_N))], \nn \\
\hat{S}_{N,12} &= \frac{1}{2N} \sumi \left(\p_{\beta}\mu_i(\hat{\beta}_N) \Sig_i^{-1}(\hat{v}_N) (Y_i-\mu_i(\hat{\beta}_N))^{\otimes2}\hat{A}_{ij} (Y_i-\mu_i(\hat{\beta}_N))\right)_{j=1}^{p_v}, \nn \\
\hat{S}_{N,22} &= \frac{1}{4N} \sumi \bigg(\tr\left(\hat{A}_{ij}(Y_i-\mu_i(\hat{\beta}_N))^{\otimes2}\right)\cdot \tr\left(\hat{A}_{ik}(Y_i-\mu_i(\hat{\beta}_N))^{\otimes2}\right) \nn \\
&{}\qquad - \tr\left(\Sig_i^{-1}(\hat{v}_N) \p_{v_j}\Sig_i(\hat{v}_N)\right)\cdot \tr\left(\Sig_i^{-1}(\hat{v}_N) \p_{v_k}\Sig_i(\hat{v}_N)\right)\bigg)_{j, k=1}^{p_v},
\nn\\
   \hat{\Gam}_{N,11} &:= \frac{1}{N}\sumi\Sig_i(\hat{v}_N)^{-1} [(\p_{\beta} \mu_i(\hat{\beta}_N))^{\otimes2}]
   \nn\\
   \hat{\Gam}_{N,22} &:= 
   \frac{1}{2N}\sumi
    \left(\tr(\Sig_i(\hat{v}_N)^{-1} (\p_{v_j} \Sig_i(\hat{v}_N)) \Sig_i(\hat{v}_N)^{-1} (\p_{v_k} \Sig_i(\hat{v}_N)))\right)_{j, k = 1}^{p_v}.
   \nn
\end{align}
Since $(\hat{\beta}_N,\hat{v}_N)=(\beta_0,v_0)+O_p(N^{-1/2})$, we have $(\hat{S}_N, \hat{\Gam}_N)\cip (S_0,\Gam_0)$. This shows the following result.

\begin{cor}
\label{hm:cor_jASN}
Under the assumptions in Theorem \ref{ti:QLA.thm1}, we have
\begin{equation}\label{hm:jASN}
\left(\hat{\Gam}_N^{-1} \hat{S}_N \hat{\Gam}_N^{-1}\right)^{-1/2}
\hat{u}_N \cil N_p(0,I_p).
\end{equation}
\end{cor}

% By \eqref{hm:jASN}, it is straightforward to construct an approximate confidence set.
When $\mu_i(\beta)=X_i\beta$, we can obtain an estimator of the $p$-value for the significance of each component of the explanatory process $X$; see also Section \ref{hm:sec_linear.case}. 
Note that the Studentization \eqref{hm:jASN} does not require us to know beforehand if the model is Gaussian or not.

\begin{rem}[Gaussian case]\label{ti:GQLME-AN-Gauss}
The previous study \cite{ImaMasTaj24} derived the local asymptotic normality and asymptotic optimality of the local maximum-likelihood estimator when the model is fully Gaussian so that the Gaussian quasi-likelihood $\mbbh_N(\theta)$ becomes the genuine log-likelihood.
As mentioned before, we have $S_{N,12}=0$ because of the symmetry of the Gaussian distribution.
% Write $\{U_{ij}\}_{j=1}^{n_i}=Y_i - \mu_i(\beta_0)$, so that $(U_{i1}, \dots, U_{in_i}) \sim N_{n_i}(0, \Sig_i(v_0))$.
% For any non-random $A_{ij} := \{a_{ijkl}\}_{k,l=1}^{n_i}$, we have
% \begin{align}
% \E\left[(Y_i - \mu_i(\beta_0))^{\otimes2}A_{ij}(Y_i - \mu_i(\beta_0))\right] 
% = \sum_{k=1}^{n_i} \sum_{l=1}^{n_i} a_{ijkl}\E[U_{ij}U_{ik}U_{il}]. \nn
% \end{align}
% From \cite[Proportion 1]{Kan2008}, we see that $\E[(Y_i - \mu_i(\beta_0))^{\otimes2}A_{ij}(Y_i - X_i\beta_0)]=0$ for each $i=1, \dots, N$ and $j=1, \dots, p_v$, hence $S_{N,12}=0$. 
Moreover, by \cite[Theorem 4.2]{MagNeu1979},
\begin{align}
S_{N,22} &= \frac{1}{N} \sumi \bigg(\frac{1}{4}\tr\left(\Sig_i^{-1}\p_{v_j}\Sig_i\right)\cdot \tr\left(\Sig_i^{-1}\p_{v_k}\Sig_i\right) \nn \\
&{}\qquad +\frac{1}{2}\tr\left(\Sig_i(v_0)^{-1}(\p_{v_j}\Sig_i)\Sig_i^{-1}(\p_{v_k}\Sig_i)\right) \nn\\
&{}\qquad 
- \frac{1}{4}\tr\left(\Sig_i^{-1}\p_{v_j}\Sig_i\right)\cdot \tr\left(\Sig_i^{-1}\p_{v_k}\Sig_i\right)\bigg)_{j, k=1}^{p_v} \nn \\
&=\frac{1}{2N} \sumi \Big(\tr\left(\Sig_i^{-1}(\p_{v_j}\Sig_i)\Sig_i^{-1}(\p_{v_k}\Sig_i)\right)\Big)_{j, k=1}^{p_v}.
\nn
\end{align}
We have $\E[\Gam_N] = S_N$ and consequently $\sqrt{N}(\tes-\tz) \cil N_p(0, S_0^{-1})$, where $S_0$ is the Fisher information matrix. It follows that, when the marginal distribution is truly Gaussian, any estimator $\tes'$ that satisfies $\sqrt{N}(\tes'-\tz)=\Gam_0^{-1} \D_N+o_p(1)$ is asymptotically efficient.
\end{rem}

\begin{rem}
Our proof based on \cite{Yos11} may apply to a broader situation where, for example, the random-effect sequences $b_1,b_2,\dots$ are not mutually independent. Under suitable additional requirements such as the strict stationarity exponential-mixing Markov property and the boundedness of moments, it would be possible to deduce similar results to Theorem \ref{ti:QLA.thm1} and Corollary \ref{hm:cor_jASN} with the same quasi-likelihood $\mbbh_N(\theta)$; this point may be related to the fact that the stationary (invariant) distribution of a Markov chain contains enough information; we refer \cite{KesSchWef01} for related details and also to \cite[Remark 2.4]{MasMerUeh24} for a related remark.
For example, one may think of the following situation: let $\{Y_i(t_{ij})\}_{j\le 24}$ denote $i$-day longitudinal data from a subject for which we obtain hourly data every day. In that case, one natural way to model the dependence of the ``daily'' data set sequence $Y_1, Y_2,\dots$ would be to make $b_1,b_2,\dots$ serially dependent. 
The same remarks apply to the stepwise procedure presented in the next section.
\end{rem}

%%%%%
%%%%%
\section{Stepwise Gaussian quasi-likelihood analysis}
\label{hm:sec_stepwise}

%%%%%
\subsection{Construction and asymptotics}
\label{ti:subsec_stepwise_construction}

The joint estimation of all parameters can be computationally demanding in our mixed-effects model setup due to the covariance function's non-linear dependence on some parameters; we will see some quantitative differences in computation time in Section \ref{sec:simulations}. 
To mitigate this issue, in this section, we will propose a stepwise estimation procedure which goes as follows:
% Our stepwise estimation policy is to estimate the mean parameter $\beta$ and the variance-covariance parameter $v$ separately and step-by-step: 
\begin{description}
\item[Stage 1] \label{ti: Stage1} 
Preliminary least-squares estimator 
$\widetilde{\beta}_{N,1} \in \argmax_{\beta \in \overline{\Theta}_\beta} \mbbh_{N,(1)}(\beta)$ for the mean, where 
\begin{align}
\mbbh_{N,(1)}(\beta) := \sumi \log \phi_{n_i}(Y_i; \mu_i(\beta), I_{n_i}),
\nn
\end{align}
which is designed based on fitting the homoscedastic Gaussian distribution.

\item[Stage 2] \label{ti: Stage2} 
Mean-adjusted covariance estimator 
$\widetilde{v}_N \in \argmax_{v \in \overline{\Theta}_v} \mbbh_{N,(2)}(v)$, where
\begin{align}
\mbbh_{N,(2)}(v) := \mbbh_N(\widetilde{\beta}_{N, 1}, v) = \sumi \log \phi_{n_i}(Y_i; \mu_i(\widetilde{\beta}_{N, 1}), \Sig_i(v)). \nn
\end{align}

\item[Stage 3] \label{ti: Stage3} 
Improved $\widetilde{\beta}_N \in \argmax_{\beta \in \overline{\Theta}_\beta} \mbbh_{N,(3)}(\beta)$, where
\begin{align}
\mbbh_{N,(3)}(\beta) := \mbbh_N(\beta, \widetilde{v}_N) = \sumi \log \phi_{n_i}(Y_i; \mu_i(\beta), \Sig_i(\widetilde{v}_N)),
\nn 
\end{align}
which is the re-weighted Gaussian fitting to take the heteroscedastic nature into account, thus improving Stage 1.
\end{description}
Let us call $\widetilde{\theta}_N := (\widetilde{\beta}_N, \widetilde{v}_N)$ the \textit{stepwise GQMLE}. 
% the stepwise estimator of the above-mentioned GQLFs. 
The estimators at Stage 1 and 3 are explicit if $\mu_i(\beta)=X_i^\top \beta$; see Section \ref{hm:sec_linear.case}. Numerical optimization in the second stage can still be time-consuming due to the non-linear dependence on $\lam$; recall the expression \eqref{hm:Sig_def}.

We will investigate the asymptotic behaviors of the stepwise GQMLE as in Theorem \ref{ti:QLA.thm1}.
Define the following variants of the quasi-score function and the quasi-observed information matrix for the GQLF $\mbbh_{N,(1)}(\beta)$ in Stage 1:
% used in the Stage $\eqref{ti: Stage1}$: 
\begin{align}
\D_{N,(1)} &:= \frac{1}{\sqrt{N}}\p_{\beta}\mbbh_{N,(1)}(\beta_0) = \frac{1}{\sqrt{N}} \sumi [\p_{\beta} \mu_i, Y_i - \mu_i], \nn \\
\Gam_{N,(1)} &:= - \frac{1}{N} \p_{\beta}^2 \mbbh_{N,(1)}(\beta_0) = \frac{1}{N} \sumi \left\{(\p_{\beta}\mu_i)^{\otimes2} - \p_{\beta}^2\mu_i[Y_i-\mu_i]\right\}.
\nn 
\end{align}
Let $\widetilde{u}_N := \sqrt{N}(\widetilde{\theta}_N - \tz)=(\sqrt{N} (\widetilde{\beta}_N - \beta_0),\sqrt{N} (\widetilde{v}_N - v_0))$.
% The asymptotic of the stepwise GQMLE is given as follows.

\begin{thm} \label{ti:QLA.thm2}
Suppose that Assumptions \ref{hm:A_1}, \ref{hm:A_2}, \ref{hm:A_3}, \ref{hm:A_4}, \ref{hm:A-S.lim}, and \ref{hm:A-Gam} hold. Moreover, suppose that there exist a positive definite matrix $\Gam_{0,(1)} \in \mbbr^{p_{\beta}} \otimes \mbbr^{p_{\beta}}$ and a measurable function $F_{1,2}(\beta)$ such that
\begin{align}
& \sup_N \sqrt{N}\left|
% \E[\Gam_{N,(1)}] 
\frac{1}{N} \sumi (\p_{\beta}\mu_i)^{\otimes2}
- \Gam_{0,(1)}\right| < \infty,
\label{ti: QLA.ass2.2}
\\
& \sup_N \sup_{\beta \in \Theta_{\beta}} \sqrt{N} \left|\frac{1}{N} \sumi (\mu_i(\beta) - \mu_i(\beta_0))^{\otimes2} - F_{1,2}(\beta)\right| < \infty,
\label{ti:QLA.ass2.6}
\end{align}
and that there exists a constant $\chi_{1}>0$ such that $F_{1,2}(\beta) \ge \chi_1|\beta-\beta_0|^2$ for every $\beta\in\Theta_\beta$.
\begin{enumerate}
\item \label{ti:QLA.thm2-2}
We have the stochastic expansion
% \begin{align}
% \sqrt{N}(\widetilde{\theta}_N - \tz) = 
% \begin{pmatrix}
% \Gam_{1,0}^{-1} \D_{N,\beta} \\
% \Gam_{3,0}^{-1} \D_{N,v} 
% \end{pmatrix}
% +O_p(N^{-1/2}). \nn
% \end{align}
% Furthermore, we have
\begin{align}
% \sqrt{N}(\widetilde{\theta}_N - \tz) 
\widetilde{u}_N 
= G_{N,1} + \frac{1}{\sqrt{N}} \widetilde{G}_{N,2} + O_p(N^{-1})
\cil N_p\left(0,\, \Gam_0^{-1} S_0 \Gam_0^{-1}\right),
\label{hm:stepwise.se}
\end{align}
where $G_{N,1}$ is the same as in \eqref{hm:def_G_1N} and 
$\widetilde{G}_{N,2} = (\widetilde{G}_{N,2,\beta},\, \widetilde{G}_{N,2,v})$ with
\begin{align} \label{ti:def_G_2N_beta}
\widetilde{G}_{N,2,\beta} &:= \Gam_{0,11}^{-1} \bigg\{\sqrt{N}(\Gam_{0,11} - \Gam_{N,11})[\Gam_{0,11}^{-1}\D_{N,\beta}] \nn\\
&{}\qquad+ \frac{1}{\sqrt{N}} \p_{\beta}\p_v \mbbh_N(\tz)[\Gam_{0,11}^{-1}\D_{N,v}] \nn \\
&\qquad + \frac{1}{N}\p_{\beta}^2\p_v\mbbh_N(\tz)[\Gam_{0,11}^{-1}\D_{N,\beta}, \Gam_{0,22}^{-1}\D_{N,v}] \nn\\
&{}\qquad + \frac{1}{2N}\p_{\beta}^3\mbbh_N(\tz)[(\Gam_{0,11}^{-1}\D_{N,\beta})^{\otimes2}]\bigg\}, 
\end{align}
\begin{align} \label{ti:def_G_2N_nu}
\widetilde{G}_{N,2,v} &:= \Gam_{0,22}^{-1}\bigg\{\sqrt{N}(\Gam_{0,22} - \Gam_{N,22})[\Gam_{0,22}^{-1}\D_{N,v}] \nn\\
&{}\qquad + \frac{1}{\sqrt{N}} \p_v\p_{\beta}\mbbh_N(\tz)[\Gam_{0,(1)}^{-1}\D_{N,(1)}]  \nn \\
&\qquad  + \frac{1}{2N}\p_v\p_{\beta}^2\mbbh_N(\tz)[(\Gam_{0,(1)}^{-1}\D_{N,(1)})^{\otimes2}] \nn\\
&{}\qquad + \frac{1}{2N}\p_v^3\mbbh_N(\tz)[(\Gam_{0,22}^{-1}\D_{N,v})^{\otimes2}] \bigg\}. 
\end{align}

\item \label{ti:QLA.thm2-1}
% For any (sufficiently large) $L_1, L_2, L_3>0$ we can find universal constants $C_{L_1}, C_{L_2}, C_{L_3}>0$ for which the rescaled estimators $\widetilde{u}_{N,1} := \sqrt{N}(\widetilde{\beta}_{N,1} - \beta_0)$, $\widetilde{u}_{N,2} := \sqrt{N}(\widetilde{v}_N - u_0)$, $\widetilde{u}_{N,3} := \sqrt{N}(\widetilde{\beta}_N - \beta_0)$ satisfy that for $r>0$
% \begin{align}
% \sup_N P[|\widetilde{u}_{N,1}| > r] \leq \frac{C_{L_1}}{r^{L_1}}, \quad \sup_N P[|\widetilde{u}_{N,2}| > r] \leq \frac{C_{L_2}}{r^{L_2}}, \quad \sup_N P[|\widetilde{u}_{N,3}| > r] \leq \frac{C_{L_3}}{r^{L_3}}. \nn 
% \end{align}
For any $L>0$, we can find a universal constant $C_L>0$ for which
\begin{align}
\sup_N \pr\left[ |\widetilde{u}_{N}| > r\right] \le \frac{C_L}{r^L},
\qquad r>0.
\nn 
\end{align}

\end{enumerate}
\end{thm}

From \eqref{ti:1st.se} and \eqref{hm:stepwise.se}, we see that the joint and stepwise GQMLEs are asymptotically first-order equivalent, that is, $|\hat{u}_N-\widetilde{u}_N| \cip 0$. 
The expressions \eqref{hm:def_G_2N} and Theorem \ref{ti:QLA.thm2} \eqref{ti:QLA.thm2-2} quantitatively show their difference in the second-order.
The proof of Theorem \ref{ti:QLA.thm2} is given in Section \ref{ti:sec_stepwise.proof}.

The Studentization \eqref{hm:jASN} remains the same.
Define $(\widetilde{S}_N,\widetilde{\Gam}_N)$ by $(\hat{S}_N,\hat{\Gam}_N)$ in Section \ref{hm:sec_an-tpe} except that all the plugged-in $\tes$ therein are replaced by $\widetilde{\theta}_N=(\widetilde{\beta}_N,\widetilde{v}_N)$.

\begin{cor}
\label{hm:cor_sASN}
\begin{equation}\label{hm:sASN}
\left(\widetilde{\Gam}_N^{-1} \widetilde{S}_N \widetilde{\Gam}_N^{-1}\right)^{-1/2}
\widetilde{u}_N \cil N_p(0,I_p).
\end{equation}
\end{cor}

\begin{rem}
In \textbf{Stage 2} in the stepwise procedure, we adopted the Gaussian density, the random function $\mbbh_{N,(2)}(v)$. We could modify it as follows:
% be further divided through optimizing the covariance-type least-squares estimator defined as a minimizer of $\| y - (\al\cdot X + b_i \cdot Z + g(\beta))\|^2$:
\begin{description}
\item[Stage 2'] $\hat{v}_N^{(0)}\in\argmin_{v}\widetilde{\mbbh}_{(2),N}(\bes^{(0)},v)$ where
\begin{equation}
\widetilde{\mbbh}_{(2),N}(\bes^{(0)},v) := \sumi \left\| \left(Y_i -\mu_i(\bes^{(0)})\right)^{\otimes 2} - \Sig_i(v) \right\|^2.
\nonumber
\end{equation}
\end{description}
This may be further divided into the two stages, which would be numerically more stable, while entailing an efficiency loss. 
Let us explain briefly. 
Recall the expression \eqref{hm:Sig_def}:
$\Sig_i(v) = Z_i \Psi(\gam) Z_i^\top + H_i(\lam, \sig^2) + \sig_{\epsilon}^2 I_{n_i}$. Since $\Sig_i(v)$ is partially linear in $c:=(\Psi(\gam),\sig_\ep^2)$. 
Regarding $\theta':=(\lam, \sig^2)$ as a known constant, we can explicitly write down the least-squares estimator of $c$ as a functional of data and $\theta'$, say $\widetilde{c}_N(\theta')$. Then, plugging-in it back to the original $\widetilde{\mbbh}_{(2),N}(\bes^{(0)},v)$, we obtain a contrast function for the parameter $\theta'$ only, say $\hat{\mbbh}_{(2),N}(\theta')$. Minimize $\hat{\mbbh}_{(2),N}$ to obtain $\tes'$, and then estimate the remaining parameter $c$ by $\hat{c}_N:=\widetilde{c}_N(\tes')$; of course, we further need the explicit form of $\gam\mapsto\Psi(\gam)$ to obtain a direct estimator $\hat{\gam}_N$ of $\gam$.
In this paper, we do not consider this point further.
\end{rem}

%%%%%
\subsection{Proof of Theorem \ref{ti:QLA.thm2}}
\label{ti:sec_stepwise.proof}

We will first prove the tail-probability estimate \eqref{ti:QLA.thm2-1} and then the second-order asymptotic expansion \eqref{ti:QLA.thm2-2}; we proceeded in reverse in Section \ref{hm:sec_joint}, but it was not essential, just because we wanted to make a natural flow by introducing several notations step by step. \rrev{The proof of Theorem \ref{ti:QLA.thm2} is fundamentally analogous to that of Theorem \ref{ti:QLA.thm1} for the joint GQMLE. However, the asymptotic expansion for the stepwise GQMLE additionally requires plugging the stochastic expansion of each stage's estimator into the expansion formula for the next stage.}

%%%%%
\subsubsection{
Tail-probability estimate
% Proof of Theorem \ref{ti:QLA.thm2} \eqref{ti:QLA.thm2-1}
}

We will separately deduce the tail-probability estimate for each component of
\begin{equation}\nn
 (\widetilde{u}_{N,1},\widetilde{u}_{N,2},\widetilde{u}_{N,3}) 
 := \left(\sqrt{N} (\widetilde{\beta}_{N,1} - \beta_0)
 ,\sqrt{N} (\widetilde{v}_N - v_0)
 , \sqrt{N} (\widetilde{\beta}_N - \beta_0) \right),
\end{equation}
%$(\widetilde{u}_{N,1},\widetilde{u}_{N,2},\widetilde{u}_{N,3}) := (\sqrt{N} (\widetilde{\beta}_{N,1} - \beta_0),\sqrt{N} (\widetilde{\beta}_N - \beta_0),\sqrt{N} (\widetilde{v}_N - v_0))$. 
again by applying the criterion given in \cite[Theorem 3]{Yos11}.

First, for $\widetilde{u}_{N,1}$, we can follow the same line as in the proof of Theorem \ref{ti:QLA.thm1} \eqref{ti:QLA.thm1-2} by replacing the variance-covariance matrix by the identity matrices $I_{n_i}$ for $i\le N$. 
It follows that $\sup_N\sup_{r>0} r^{L} \pr\left[ |\widetilde{u}_{N,1}| > r\right] < \infty$, therefore, in particular $\sup_N\E[|\widetilde{u}_{N,1}|^K]<\infty$ for every $K>0$, which will be used subsequently.

Turning to $\widetilde{u}_{N,2}$, we apply the Taylor expansion
\begin{align}
\p_v^k\mbbh_{N,(2)}(v) &= \p_v^k\mbbh_N(\widetilde{\beta}_{N,1},v)\nn \\
&= \p_v^k\mbbh_{N}(\beta_0,v) + \left(\int_0^1 \frac{1}{\sqrt{N}}\p_{\beta}\p_v^k \mbbh_{N}(\beta_0 + s(\widetilde{\beta}_{N,1} - \beta_0), v)ds\right) [\widetilde{u}_{N,1}] 
\nn
\end{align}
for $k = 0, 1, 2, 3$. 
As in the proof of Theorem \ref{ti:QLA.thm1} \eqref{ti:QLA.thm1-2}, the random functions required for proving the tail probability evaluation in Stage 2 are given as follows:
\begin{align}
\D_{N,(2)} &:= \frac{1}{\sqrt{N}} \p_v\mbbh_{N,(2)}(v_0) 
= \frac{1}{\sqrt{N}} \p_v\mbbh_N(\widetilde{\beta}_{N,1}, v_0) 
\nn \\
% &= \frac{1}{\sqrt{N}} \p_v\mbbh_{N}(\beta_0, v_0) + \left(\int_0^1 \frac{1}{N}\p_{\beta}\p_v \mbbh_{N}(\beta_0 + s(\widetilde{\beta}_{N,1} - \beta_0), v_0)ds\right) [\sqrt{N}(\widetilde{\beta}_{N,1} - \beta_0)] \nn \\
&= \D_{N,v} + \left\{\left(\int_0^1 \frac{1}{N}\p_{\beta}\p_v \mbbh_{N}(\beta_0 + s(\widetilde{\beta}_{N,1} - \beta_0), v_0)ds\right) [\widetilde{u}_{N,1}]\right\},
\label{hm:202408082154-1} \\
\Gam_{N,(2)} &:= -\frac{1}{N}\p_v^2\mbbh_{N,(2)}(v_0) 
= -\frac{1}{N} \p_v^2 \mbbh_N(\widetilde{\beta}_{N,1}, v_0) 
\label{hm:202408082154-2} \\
% &= -\frac{1}{N}\p_v^2 \mbbh_N(\beta_0, v_0) - \frac{1}{\sqrt{N}}\left(\int_0^1 \frac{1}{N}\p_{\beta}\p_v^2 \mbbh_N(\beta_0 + s(\widetilde{\beta}_{N,1} - \beta_0), v_0)ds\right) [\sqrt{N}(\widetilde{\beta}_{N,1} - \beta_0)] \nn \\
&= \Gam_{N,22} - \frac{1}{\sqrt{N}}\left\{\left(\int_0^1 \frac{1}{N}\p_{\beta}\p_v^2 \mbbh_{N}(\beta_0 + s(\widetilde{\beta}_{N,1} - \beta_0), v_0)ds\right) [\widetilde{u}_{N,1}]\right\}, \nn \\
\mbby_{N,(2)}(v) &:= \frac{1}{N}(\mbbh_{N,(2)}(v) - \mbbh_{N,(2)}(v_0))
= \frac{1}{N}(\mbbh_N(\widetilde{\beta}_{N,1}, v) - \mbbh_N(\widetilde{\beta}_{N,1}, v_0)) \nn \\ 
&= 
\mbby_{N,v}(v) 
% \frac{1}{N}(\mbbh_N(\beta_0, v) - \mbbh_N(\beta_0, v_0)) \nn \\
+ \frac{1}{\sqrt{N}}\bigg\{\left(\int_0^1 \frac{1}{N}\p_{\beta} \mbbh_{N}(\beta_0 + s(\widetilde{\beta}_{N,1} - \beta_0), v)ds\right) [\widetilde{u}_{N,1}] \nn \\
&\qquad\quad - \left(\int_0^1 \frac{1}{N}\p_{\beta} \mbbh_{N}(\beta_0 + s(\widetilde{\beta}_{N,1} - \beta_0), v_0)ds\right) [\widetilde{u}_{N,1}]\bigg\},
\label{hm:202408082154-3}
% \\
% &=\mbby_{N,v}(v) + \frac{1}{\sqrt{N}}\left(\int_0^1 \frac{1}{N}\p_{\beta} \mbbh_{N}(\beta_0 + s(\widetilde{\beta}_{N,1} - \beta_0), v)ds\right) [\sqrt{N}(\widetilde{\beta}_{N,1} - \beta_0)], 
% \nn
\end{align}
where $\mbby_{N,v}(v) := N^{-1} (\mbbh_N(\beta_0,v) - \mbbh_N(\beta_0,v_0))$, and finally,
\begin{align}
\frac{1}{N}\p_v^3\mbbh_{N,(2)}(v) 
&= \frac{1}{N} \p_v^3\mbbh_N(\beta_0, v) \nn\\
&{}
+ \frac{1}{\sqrt{N}}\left\{\left(\int_0^1 \frac{1}{N}\p_{\beta} \p_v^3 \mbbh_{N}(\beta_0 + s(\widetilde{\beta}_{N,1} - \beta_0), v)ds\right) [\widetilde{u}_{N,1}]\right\}.
\label{hm:202408082154-4}
\end{align}
As in Section \ref{hm:sec_joint}, under the present assumptions, we can show that the curly-bracket parts $\{\dots\}$ in the expressions \eqref{hm:202408082154-1} to \eqref{hm:202408082154-4} are all $L^K$-bounded for every $K>0$ uniformly in $v$, enabling us to proceed with the moment estimates as we have done for $\D_N$, $\Gam_N$, $\mbby_N(\theta)$, and $\p_\theta^3\mbbh_N(\theta)$ in Section \ref{hm:sec_joint}.
Thus, we proved Theorem \ref{ti:QLA.thm2} \eqref{ti:QLA.thm2-1} for $\widetilde{u}_{N,2}$, followed by $\sup_N\E[|\widetilde{u}_{N,2}|^K]<\infty$ for every $K>0$.
%同時推定時の\eqref{ti:QLA.thm1.3}について$\p_{\theta}^3\mbbh_N(\theta)/N$を$\p_v^3\mbbh_{N}(\beta_0,v)/N$に, \eqref{ti:QLA.thm1.4}について$\D_{N}$を$\D_{N,v}$に, \eqref{ti:QLA.thm1.5}について$\Gam_N$を$\Gam_{N,3}$に, \eqref{ti:QLA.thm1.6}について$\mbby_N(\theta)$を$\mbby_{N,v}(v)$に置き換えて成立することが言えれば良い. 

Finally, as for $\widetilde{u}_{N,3}$, we note that
\begin{align}
\p_{\beta}^k\mbbh_{N,(3)}(\beta) &= \p_{\beta}^k\mbbh_N(\beta,\widetilde{v}_N) \nn \\
&= \p_{\beta}^k\mbbh_N(\beta,v_0) + \left(\int_0^1 \frac{1}{\sqrt{N}}\p_v\p_{\beta}^k \mbbh_N(\beta, v_0+s(\widetilde{v}_N-v_0))ds\right) [\widetilde{u}_{N,2}] \nn
\end{align}
for $k = 0, 1, 2, 3$.
% For convenience, let us use the generic notation $U_n(\theta)$ for any random variables such that $\sup_N \E[\sup_\theta |U_n(\theta)|^K]<\infty$ for every $K>0$.
As before, we have
%ここで, $k = 1$に対して$\p_v\p_{\beta}\mbbh_N(\beta_0, \widetilde{\widetilde{v}}_N)/\sqrt{N}$ ($\widetilde{\widetilde{v}}_N$は$\widetilde{v}_N$と$v_0$を結ぶ線分上の任意の点) は中心極限定理より$O_p(1)$である. 一方, $k \neq 1$については, $\p_v\p_{\beta}^k\mbbh_N(\beta, \widetilde{\widetilde{v}}_N)/N$は$Y_i$のモーメント条件, および$\mu_i(\beta), \Sig_i(v)$の仮定より$O_p(1)$が得られる. 
\begin{align}
    \D_{N,(3)} &:= \frac{1}{\sqrt{N}}\p_{\beta}\mbbh_{N,(3)}(\beta_0) 
= \frac{1}{\sqrt{N}}\p_{\beta}\mbbh_N(\beta_0,\widetilde{v}_N) \nn \\
&= \D_{N,\beta} + \left(\int_0^1 \frac{1}{N}\p_v\p_{\beta} \mbbh_N(\beta_0, v_0+s(\widetilde{v}_N-v_0))ds\right) [\widetilde{u}_{N,2}]
\nn
\end{align}
and $\sup_N \E[|\D_{N,(3)}-\D_{N,\beta}|^K]<\infty$ for every $K>0$.
In a similar fashion, 
\begin{equation}\nn
    \Gam_{N,(3)} := -\frac{1}{N}\p_{\beta}^2\mbbh_{N,(3)}(\beta_0) 
 = -\frac{1}{N}\p_{\beta}^2 \mbbh_N(\beta_0, \widetilde{v}_N) 
\end{equation}
% $\Gam_{N,(3)} := -N^{-1}\p_{\beta}^2\mbbh_{N,(3)}(\beta_0) 
% = - N^{-1}\p_{\beta}^2 \mbbh_N(\beta_0, \widetilde{v}_N)$ 
satisfies that $\sup_N \E[|\Gam_{N,(3)}-\Gam_{N,11}|^K]<\infty$ for every $K>0$. 
% \begin{align}
% \Gam_{N,(3)} &:= -\frac{1}{N}\p_{\beta}^2\mbbh_{N,(3)}(\beta_0) 
% = -\frac{1}{N}\p_{\beta}^2 \mbbh_N(\beta_0, \widetilde{v}_N) \nn \\
% &= -\frac{1}{N}\p_{\beta}^2\mbbh_N(\beta_0, v_0) - \frac{1}{\sqrt{N}}\left(\int_0^1 \frac{1}{N}\p_v\p_{\beta}^2\mbbh_N(\beta_0, v_0+s(\widetilde{v}_N-v_0))ds\right) [\sqrt{N}(\widetilde{v}_N-v_0)] \nn \\
% &= \Gam_{N,11} - \frac{1}{\sqrt{N}}\left(\int_0^1 \frac{1}{N}\p_v\p_{\beta}^2\mbbh_N(\beta_0, v_0+s(\widetilde{v}_N-v_0))ds\right) [\sqrt{N}(\widetilde{v}_N-v_0)], \nn
% \end{align}
Also, as in $\mbby_{N,(2)}(v)$ in the previous paragraph,
\begin{equation}
\mbby_{N,(3)}(\beta) := \frac{1}{N}(\mbbh_{N,(3)}(\beta) - \mbbh_N(\beta_0,\widetilde{v}_N)) = \frac{1}{N}(\mbbh_N(\beta, \widetilde{v}_N) - \mbbh_N(\beta_0,\widetilde{v}_N)) \nn
\end{equation}
satisfies that $\sup_N \E[\sup_\beta|\mbby_{N,(3)}(\beta) - \mbby_{N,\beta}(\beta)|^K]<\infty$ for every $K>0$, where
% $\mbby_{N,\beta}(\beta) := N^{-1}(\mbbh_N(\beta,v_0) - \mbbh_N(\beta_0,v_0))$.
\begin{equation}\nn
    \mbby_{N,\beta}(\beta) := \frac1N \big(\mbbh_N(\beta,v_0) - \mbbh_N(\beta_0,v_0)\big).
\end{equation}
% \begin{align}
% \mbby_{N,(3)}(\beta) &:= \frac{1}{N}(\mbbh_{N,(3)}(\beta) - \mbbh_N(\beta_0,v_0)) \nn \\
% &= \frac{1}{N}(\mbbh_N(\beta, \widetilde{v}_N) - \mbbh_N(\beta_0,v_0)) \nn \\
% &= \frac{1}{N}(\mbbh_N(\beta, v_0) - \mbbh_N(\beta_0, v_0)) + \frac{1}{\sqrt{N}}\left(\int_0^1 \frac{1}{N}\p_v\mbbh_N(\beta, v_0+s(\widetilde{v}_N-v_0))ds\right) [\sqrt{N}(\widetilde{v}_N-v_0)] \nn \\
% &= \mbby_{N,\beta}(\beta) + \frac{1}{\sqrt{N}}\left(\int_0^1 \frac{1}{N}\p_v\mbbh_N(\beta, v_0+s(\widetilde{v}_N-v_0))ds\right) [\sqrt{N}(\widetilde{v}_N-v_0)], \nn
% \end{align}
Moreover, we have 
\begin{equation}\nn
    \sup_N \E\left[\sup_\beta \left|\frac{1}{N}\p_{\beta}^3\mbbh_{N,(3)}(\beta) - \frac{1}{N}\p_{\beta}^3\mbbh_N(\beta, v_0) \right|^K\right] < \infty
\end{equation}
for every $K>0$.
% \begin{align}
% \frac{1}{N}\p_{\beta}^3\mbbh_{N,(3)}(\beta) &= \frac{1}{N}\p_{\beta}^3\mbbh_N(\beta, v_0) + \left(\int_0^1 \frac{1}{N}\p_v\p_{\beta}^3\mbbh_N(\beta, v_0+s(\widetilde{v}_N-v_0))ds\right) [\sqrt{N}(\widetilde{v}_N-v_0)], \nn
% \end{align}
With these moment estimates, we obtain Theorem \ref{ti:QLA.thm2} \eqref{ti:QLA.thm2-1} for $\widetilde{u}_{N,3}$.

%%By Lemma \ref{ti:QLA.lem4} and \ref{ti:QLA.lem5}, we obtain the following theorem. 
%\begin{thm}\label{ti:QLA.thm2}
%For any $L>0$ we can find a universal constant $C_L>0$ for which the rescaled estimator $\hat{u}_N:=\sqrt{N}(\tes-\tz)$ satisfies that
%\begin{equation}
%\sup_N \pr\left(|\hat{u}_N|>r\right) \le \frac{C_L}{r^L}, \qquad r>0.
%\label{ti:tail.prob.estimate}
%\end{equation}
%\end{thm}
%
%\begin{proof}
%Define
%\begin{equation}
%V_N(r; \tz) = \left\{u \in U_N(\tz); r \leq |u|\right\}, \nn
%\end{equation}
%where
%\begin{equation}
%U_N(\tz) = \left\{u \in \mbbr^p; \tz + N^{-1/2} u \in \Theta\right\}. \nn
%\end{equation}
%Theorem \ref{hm:QLA.thm1} ensures conditions of Theorem 3 in \cite{Yos11}. For any sufficiently large $L > 0$, $r > 0$, and $N \in \mbbn$, we have
%\begin{equation}
%P\left[\sup_{u \in V_N(r; \tz)} \mbbz_N(u; \tz) \geq e^{-r}\right] \leq \frac{C_L}{r^L}, \nn
%\end{equation}
%where $C_L > 0$ is a constant, 
%\begin{equation}
%\mbbz_N(u; \tz) = \exp\left\{\mbbh_N(\tz + N^{-1/2}u) - \mbbh_N(\tz)\right\}. \nn
%\end{equation}
%From Theorem 3 in \cite{Yos11}, we have
%\begin{align}
%P\left(|\sqrt{N}(\tes - \tz)| \geq r\right) &\leq P\left(\sup_{|u| \geq r} \left\{\mbbh_N(\tz + N^{-1/2}u) - \mbbh_N(\tz)\right\} \geq 0\right) \nn \\
%&\leq P\left(\sup_{|u| \geq r} \left\{\mbbh_N(\tz + N^{-1/2}u) - \mbbh_N(\tz) \right\} \geq - r\right) \nn \\
%&\leq P\left(\sup_{|u| \geq r} \mbbz_N(u; \tz) \geq e^{-r}\right) \nn \\
%&\leq \frac{C_L}{r^L}. \nn
%\end{align}
%Therefore, $\sup_N P(\sqrt{N}(\tes - \tz) \geq r) \leq C_L / r^L$. 
%\end{proof}

%%%%%
\subsubsection{
Stochastic expansion and asymptotic normality
% Proof of Theorem \ref{ti:QLA.thm2} \eqref{ti:QLA.thm2-2}
}

We will look at $\widetilde{u}_{N,2}$ and $\widetilde{u}_{N,3}$ separately. The fact $\widetilde{u}_{N}=O_p(1)$ derived in the previous subsection will be used repeatedly without mention.

As in Lemma \ref{ti:boundary of three and forth dev}, we can show that
$\sup_N \E[\sup_{\beta}|N^{-1} \p_{\beta}^4 \mbbh_{N,(3)}(\beta)|^K ] < \infty$ for all $K > 0$. Then, we expand the score functions in Stage 3 around $\beta_0$:
\begin{align}
\frac{1}{\sqrt{N}} \p_{\beta} \mbbh_{N,(3)}(\widetilde{\beta}_{N}) 
&= \frac{1}{\sqrt{N}} \p_{\beta} \mbbh_N(\beta_0, \widetilde{v}_N) + \frac{1}{N}\p_{\beta}^2 \mbbh_N(\beta_0, \widetilde{v}_N)[\widetilde{u}_{N,3}] \nn \\
&\qquad + \frac{1}{\sqrt{N}} \frac{1}{2N} \p_{\beta}^3 \mbbh_N(\beta_0, \widetilde{v}_N)
[\widetilde{u}_{N,3}^{\otimes 2}] + O_p(N^{-1}).
\nn%\label{hm:stepwise-1'}
\end{align}
Since $\widetilde{\beta}_{N} \cip \beta_0$ with the limit lying in the interior of the parameter space, we have $N^\kappa \p_{\beta} \mbbh_{N,(3)}(\widetilde{\beta}_{N})=o_p(1)$ for any $\kappa\in\mbbr$, in particular, $N^{-1/2}\p_{\beta} \mbbh_{N,(3)}(\widetilde{\beta}_{N})=O_p(N^{-1})$. 
This gives
\begin{align}
-\frac{1}{N}\p_{\beta}^2 \mbbh_N(\beta_0, \widetilde{v}_N)[\widetilde{u}_{N,3}]
&= 
\frac{1}{\sqrt{N}} \p_{\beta} \mbbh_N(\beta_0, \widetilde{v}_N) \nn\\
&{}\qquad 
+ \frac{1}{\sqrt{N}} \frac{1}{2N} \p_{\beta}^3 \mbbh_N(\beta_0, \widetilde{v}_N)
[\widetilde{u}_{N,3}^{\otimes 2}] + O_p(N^{-1}).
\label{hm:stepwise-1'}
\end{align}

First, we note the first-order expansion.
Obviously,
\begin{align}
\frac{1}{\sqrt{N}}\p_{\beta}\mbbh_N(\beta_0,\widetilde{v}_N) &= \D_{N,\beta}+O_p(N^{-1/2}),
\nn\\
-\frac1N \p_{\beta}^2 \mbbh_N(\beta_0, \widetilde{v}_N)&=\Gam_{N,11}+O_p(N^{-1/2})=\Gam_{0,11}+O_p(N^{-1/2}).
\nn
\end{align}
By $\D_{N,\beta}=O_p(1)$, we conclude that
\begin{align}
%\sqrt{N}(\widetilde{\beta}_N - \beta_0)
\widetilde{u}_{N,3}
&=\left( \Gam_{0,11} + O_p(N^{-1/2}) \right)^{-1} \left(\D_{N,\beta} + O_p(N^{-1/2}) \right)
\nn\\
&= \Gam_{0,11}^{-1} \D_{N, \beta} + O_p(N^{-1/2}).
\label{hm:stepwise-b1}
\end{align}
Similarly,
\begin{align}
%\sqrt{N}(\widetilde{v}_N - v_0) 
\widetilde{u}_{N,2} &= \Gam_{0,22}^{-1} \D_{N,v} + O_p(N^{-1/2}).
\label{hm:stepwise-v1}
\end{align}
It follows that $\widetilde{u}_N = G_{N,1} + O_p(N^{-1/2})$.

Turning to the second-order expansion, we note by \eqref{hm:stepwise-v1},
\begin{align}
\frac{1}{\sqrt{N}} \p_{\beta} \mbbh_N(\beta_0, \widetilde{v}_N) 
&= \D_{N,\beta} + \frac{1}{N}\p_v\p_{\beta}\mbbh_N(\tz)[\widetilde{u}_{N,2}] \nn \\
&\qquad + \frac{1}{2N} \frac{1}{\sqrt{N}} \p_v^2 \p_{\beta} \mbbh_N(\tz)[\widetilde{u}_{N,2}^{\otimes2}] + O_p(N^{-1})
\nn \\
&= \D_{N,\beta} + \frac{1}{\sqrt{N}} \frac{1}{\sqrt{N}}\p_v\p_{\beta}\mbbh_N(\tz)[\Gam_{0,22}^{-1} \D_{N, v}] \nn \\
&\qquad + \frac{1}{2N} \frac{1}{\sqrt{N}}\p_v^2\p_{\beta} \mbbh_N(\tz)[(\Gam_{0,22}^{-1}\D_{N,v})^{\otimes2}] + O_p(N^{-1})
\nn\\
&= \D_{N,\beta} + \frac{1}{\sqrt{N}} \frac{1}{\sqrt{N}}\p_v\p_{\beta}\mbbh_N(\tz)[\Gam_{0,22}^{-1} \D_{N, v}] + O_p(N^{-1}),
\nn
\end{align}
and similarly,
\begin{align}
& -\frac{1}{N}\p_{\beta}^2\mbbh_N(\beta_0, \widetilde{v}_N) 
%&= \frac{1}{N}\p_{\beta}^2\mbbh_N(\beta_0, v_0) + \frac{1}{\sqrt{N}} \frac{1}{N} \p_v\p_{\beta}^2 \mbbh_N(\beta_0, v_0)[\widetilde{u}_{N,2}] + O_p(N^{-1}), \nn \\
\nn\\
&= \Gam_{N,11} - \frac{1}{\sqrt{N}} \frac{1}{N} \p_v\p_{\beta}^2 \mbbh_N(\tz)[\Gam_{0,22}^{-1}\D_{N,v}] + O_p(N^{-1})
\nn \\
&= \Gam_{0,11} - \frac{1}{\sqrt{N}} \left(
-\sqrt{N}(\Gam_{N,11}-\Gam_{0,11})
+\frac{1}{N} \p_v\p_{\beta}^2 \mbbh_N(\tz)[\Gam_{0,22}^{-1}\D_{N,v}]
\right)+ O_p(N^{-1}),
\nn
\end{align}
and
\begin{align}
\frac{1}{\sqrt{N}} \frac{1}{2N} \p_{\beta}^3\mbbh_N(\beta_0, \widetilde{v}_N) 
&= \frac{1}{\sqrt{N}} \frac{1}{2N} \p_{\beta}^3\mbbh_N(\tz) + O_p(N^{-1}).
\nn
\end{align}
Substituting these three expressions in \eqref{hm:stepwise-1'} and then arranging them, we obtain
\begin{align}
\widetilde{u}_{N,3}
&= \Gam_{0,11}^{-1}\D_{N,\beta} + \frac{1}{\sqrt{N}}\Gam_{0,11}^{-1}\bigg\{\sqrt{N}(\Gam_{0,11}-\Gam_{N,11})[\Gam_{0,11}^{-1}\D_{N,\beta}] \nn\\
&{}\qquad \qquad + \frac{1}{\sqrt{N}}\p_v\p_{\beta}\mbbh_{N}(\tz)[\Gam_{0,22}^{-1}\D_{N,v}] \bigg\} \nn \\
&\qquad + \frac{1}{2\sqrt{N}}\Gam_{0,11}^{-1}\bigg\{\frac{1}{N}\p_v^2\p_{\beta}\mbbh_{N}(\tz)[(\Gam_{0,22}^{-1}\D_{N,v})^{\otimes2}] 
\nn\\
&{}\qquad\qquad + \frac{2}{N} \p_v\p_{\beta}^2\mbbh_{N}(\tz)[\Gam_{0,11}^{-1}\D_{N,\beta}, \Gam_{0,22}^{-1}\D_{N,v}] \nn \\
&\qquad \qquad + \frac{1}{N}\p_{\beta}^3\mbbh_{N}(\tz)[(\Gam_{0,11}^{-1}\D_{N,\beta})^{\otimes2}]\bigg\} + O_p(N^{-1}) \nn \\
%&= \Gam_{0,11}^{-1}\D_{N,\beta} + \frac{1}{\sqrt{N}}\Gam_{0,11}^{-1}\left\{\sqrt{N}(\Gam_{0,11}-\Gam_{N,11})[\Gam_{0,11}^{-1}\D_{N,\beta}] + \frac{1}{\sqrt{N}}\p_{\beta}\p_v\mbbh_{N}(\beta_0,v_0)[\Gam_{0,22}^{-1}\D_{N,v}] \right\} \nn \\
%&\qquad + \frac{1}{2\sqrt{N}}\Gam_{0,11}^{-1}\left\{\frac{2}{N} \p_{\beta}^2\p_v\mbbh_{N}(\beta_0,v_0)[\Gam_{0,11}^{-1}\D_{N,\beta}, \Gam_{0,22}^{-1}\D_{N,v}] + \frac{1}{N}\p_{\beta}^3\mbbh_{N}(\beta_0,v_0)[(\Gam_{0,11}^{-1}\D_{N,\beta})^{\otimes2}]\right\} \nn \\
%&\qquad + O_p(N^{-1}) \nn\\
% &= \Gam_{0,11}^{-1}\D_{N,\beta} + \frac{1}{\sqrt{N}}\Gam_{0,11}^{-1}\bigg\{
% \sqrt{N}(\Gam_{0,11}-\Gam_{N,11})[\Gam_{0,11}^{-1}\D_{N,\beta}] + \frac{1}{\sqrt{N}}\p_{\beta}\p_v\mbbh_{N}(\tz)[\Gam_{0,22}^{-1}\D_{N,v}]
% \nn \\
% &\qquad + 
% \frac{1}{N} \p_{\beta}^2\p_v\mbbh_{N}(\tz)[\Gam_{0,11}^{-1}\D_{N,\beta}, \Gam_{0,22}^{-1}\D_{N,v}] + \frac{1}{2N}\p_{\beta}^3\mbbh_{N}(\tz)[(\Gam_{0,11}^{-1}\D_{N,\beta})^{\otimes2}]
% \bigg\} \nn \\
% &\qquad + O_p(N^{-1})
% \nn\\
&= \Gam_{0,11}^{-1}\D_{N,\beta} + \frac{1}{\sqrt{N}}\widetilde{G}_{N,2,\beta} + O_p(N^{-1}).
\label{hm:stepwise-p2'}
\end{align}

As for the stochastic expansion of $\widetilde{v}_N$, we calculate the stochastic expansion of the estimator $\widetilde{\beta}_{N,1}$ in Stage 1 up to $O_p(N^{-1/2})$:
\begin{align}
\widetilde{u}_{N,1} = \Gam_{0,(1)}^{-1}\D_{N,(1)} + O_p(N^{-1/2}).
\nn
\end{align}
In the present case, we have
\begin{align}
\frac{1}{\sqrt{N}}\p_v\mbbh_N(\widetilde{\beta}_{N,1},v_0) 
% &= \frac{1}{\sqrt{N}}\p_v\mbbh_N(\tz) + \frac{1}{N}\p_{\beta}\p_v\mbbh_N(\beta_0, v_0)[\widetilde{u}_{N,1}] \nn \\
% &\qquad + \frac{1}{2\sqrt{N}} \frac{1}{N}\p_{\beta}^2\p_v\mbbh_N(\tz)[\widetilde{u}_{N,1}^{\otimes2}] \nn \\
% &\qquad + \frac{1}{N} \left(\iiint_{(0,1)^3} \frac{s^2t}{N}\p_{\beta}^3\p_v \mbbh_N(\beta_0+stu(\widetilde{\beta}_{N,1} - \beta_0), v_0)dsdtdu\right)[\widetilde{u}_{N,1}^{\otimes3}] \nn \\
&= \frac{1}{\sqrt{N}}\p_v\mbbh_N(\tz) + \frac{1}{N}\p_{\beta}\p_v\mbbh_N(\tz)[\Gam_{0,(1)}^{-1}\D_{N,(1)}] \nn \\
&\qquad + \frac{1}{2\sqrt{N}} \frac{1}{N}\p_{\beta}^2\p_v\mbbh_N(\tz)[(\Gam_{0,(1)}^{-1}\D_{N,(1)})^{\otimes2}] + O_p(N^{-1}), \nn
\end{align}
\begin{align}
\frac{1}{N}\p_v^2\mbbh_N(\widetilde{\beta}_{N,1},v_0) 
% &= \frac{1}{N}\p_v^2\mbbh_N(\tz) + \frac{1}{\sqrt{N}} \frac{1}{N}\p_\beta\p_v^2\mbbh_N(\tz)[\widetilde{u}_{N,1}] \nn \\
% &\qquad + \frac{1}{N}\left(\iint_{(0,1)^2}\frac{s}{N}\p_{\beta}^2\p_v^2 \mbbh_N(\beta_0+st(\widetilde{\beta}_{N,1}-\beta_0), v_0)dsdt\right)[\widetilde{u}_{N,1}^{\otimes2}] \nn \\
&= \frac{1}{N}\p_v^2\mbbh_N(\tz) + \frac{1}{\sqrt{N}} \frac{1}{N}\p_\beta\p_v^2\mbbh_N(\tz)[\Gam_{0,(1)}^{-1}\D_{N,(1)}] + O_p(N^{-1}), \nn
\end{align}
and
\begin{align}
\frac{1}{\sqrt{N}} \frac{1}{N}\p_v^3\mbbh_N(\widetilde{\beta}_{N,1},v_0) 
% &= \frac{1}{\sqrt{N}} \frac{1}{N}\p_v^3\mbbh_N(\tz) 
% \nn\\
% &{}\qquad + \frac{1}{N} \left(\int_0^1 \frac{1}{N}\p_{\beta}\p_v^3\mbbh_N(\beta_0+s(\widetilde{\beta}_{N,1} - \beta_0))ds\right)[\widetilde{u}_{N,1}] \nn \\
&= \frac{1}{\sqrt{N}} \frac{1}{N}\p_v^3\mbbh_N(\tz) + O_p(N^{-1}). \nn
\end{align}
Using these expressions, we can proceed as in the case of $\widetilde{\beta}_N$ to arrive at the stochastic expansion:
\begin{align}
\widetilde{u}_{N,2} 
&= \Gam_{0,22}^{-1}\D_{N,v} + \frac{1}{\sqrt{N}}\Gam_{0,22}^{-1}\bigg(\sqrt{N}(\Gam_{0,22}-\Gam_{N,22})[\Gam_{0,22}^{-1}\D_{N,v}] 
\nn\\
&{}\qquad\qquad + \frac{1}{\sqrt{N}}\p_{\beta}\p_v\mbbh_{N}(\beta_0,v_0)[\Gam_{0,(1)}^{-1}\D_{N,(1)}]\bigg) \nn \\
&\qquad + \frac{1}{2\sqrt{N}}\Gam_{0,22}^{-1}\bigg(\frac{1}{N}\p_{\beta}^2\p_v\mbbh_{N}(\beta_0,v_0)[(\Gam_{0,(1)}^{-1}\D_{N,(1)})^{\otimes2}] 
\nn\\
&{}\qquad\qquad + \frac{2}{N}\p_{\beta}\p_v^2\mbbh_{N}(\beta_0,v_0)[\Gam_{0,(1)}^{-1}\D_{N,(1)}, \Gam_{0,22}^{-1}\D_{N,v}] \nn \\
&\qquad + \frac{1}{N}\p_v^3\mbbh_{N}(\beta_0,v_0)[(\Gam_{0,22}^{-1}\D_{N,v})^{\otimes2}]\bigg) + O_p(N^{-1}) \nn \\
% &= \Gam_{0,22}^{-1}\D_{N,v} + \frac{1}{\sqrt{N}}\Gam_{0,22}^{-1}\left\{\sqrt{N}(\Gam_{0,22}-\Gam_{N,22})[\Gam_{0,22}^{-1}\D_{N,v}] + \frac{1}{\sqrt{N}}\p_{\beta}\p_v\mbbh_{N}(\beta_0,v_0)[\Gam_{0,(1)}^{-1}\D_{N,(1)}]\right\} \nn \\
% &\qquad + \frac{1}{2\sqrt{N}}\Gam_{0,22}^{-1}\left\{\frac{1}{N}\p_{\beta}^2\p_v\mbbh_{N}(\beta_0,v_0)[(\Gam_{0,(1)}^{-1}\D_{N,(1)})^{\otimes2}] + \frac{1}{N}\p_v^3\mbbh_{N}(\beta_0,v_0)[(\Gam_{0,22}^{-1}\D_{N,v})^{\otimes2}]\right\} \nn \\
% &\qquad + O_p(N^{-1}). \nn\\
% &= \Gam_{0,22}^{-1}\D_{N,v} 
% + \frac{1}{\sqrt{N}}\Gam_{0,22}^{-1}\bigg(
% \sqrt{N}(\Gam_{0,22}-\Gam_{N,22})[\Gam_{0,22}^{-1}\D_{N,v}] 
% \nn\\
% &{}\qquad\qquad + \frac{1}{\sqrt{N}}\p_{\beta}\p_v\mbbh_{N}(\beta_0,v_0)[\Gam_{0,(1)}^{-1}\D_{N,(1)}]
% + \frac{1}{2N}\p_{\beta}^2\p_v\mbbh_{N}(\beta_0,v_0)[(\Gam_{0,(1)}^{-1}\D_{N,(1)})^{\otimes2}] 
% \nn \\
% &{}\qquad\qquad 
% + \frac{1}{2N}\p_v^3\mbbh_{N}(\beta_0,v_0)[(\Gam_{0,22}^{-1}\D_{N,v})^{\otimes2}]\bigg) + O_p(N^{-1})
% \nn\\
&= \Gam_{0,22}^{-1}\D_{N,v} + \frac{1}{\sqrt{N}}\widetilde{G}_{N,2,v} + O_p(N^{-1}).
\label{hm:stepwise-p3'}
\end{align} 
Combining \eqref{hm:stepwise-p2'} and \eqref{hm:stepwise-p3'} completes the proof of Theorem \ref{ti:QLA.thm2} \eqref{ti:QLA.thm2-2}.

%%%%%
%%%%%
\section{Remarks on the partially linear case}
\label{hm:sec_linear.case}

In this section, we take a closer look at some of the assumptions and statements in Theorem \ref{ti:QLA.thm2} in the original model, that is, \eqref{hm:target.model} and \eqref{hm:model_Yij} where
\begin{align}
\mu_i(\beta) &= X_i \beta, \nn\\
\Sig_i(v) &= Z_i \Psi(\gam) Z_i^\top + H_i(\lam, \sig^2) + \sig_{\epsilon}^2 I_{n_i}
\nn%\label{hm:Sig_def}
\end{align}
with the expression \eqref{hm:H_form}.
We have
\begin{align}
    \mbbh_N(\theta)
    &= C'_N -\frac12 \sumi\left(
        \log|\Sig_i(v)| + \Sig_i(v)^{-1}\left[(Y_i - X_i\beta)^{\otimes 2}\right]
    \right),
\end{align}
where $C'_N$ is a constant independent of $\theta$.
Some entries of $\p_\theta^k\mbbh_N(\theta)$ can be simplified: for $l\ge 0$,
\begin{align}
    \p_v^l\p_\beta\mbbh_N(\theta)
    &= \sumi X_i^\top \p_v^l\left(\Sig_i(v)^{-1}\right)(Y_i-X_i\beta),
    \nn\\
    \p_v^l\p^2_\beta\mbbh_N(\theta)
    &= -\sumi X_i^\top \p_v^l\left(\Sig_i(v)^{-1}\right) X_i,
    \nn\\
    \p_v^l\p^3_\beta\mbbh_N(\theta) &\equiv 0.
    \nn
\end{align}
Still, the forms of the partial derivatives of $\mbbh_N$ for $v$ are somewhat messy. But the cross partial derivatives of $\Sig_i(v)$ with respect to the variables $\gam$, $(\lam,\sig^2)$, and $\sig_\ep^2$ vanishes, and $\p_{\sig^2}^k \Sig_i(v) =0$ and $\p_{\sig_\ep^2}^k \Sig_i(v) =0$ for $k\ge 2$.

Concerning the stepwise GQMLE, the ones in Stage 1 and 3 are explicitly given as
\begin{align}
\widetilde{\beta}_{N,1} &= \left(\sumi X_i^\top X_i\right)^{-1}\sumi X_i^\top Y_i,
\nn\\
\widetilde{\beta}_{N} &= \left(\sumi X_i^\top X_i\right)^{-1}\sumi X_i^\top \Sig_i(\widetilde{v}_N)^{-1}Y_i,
\nn
\end{align}
while $\widetilde{v}_N$ still requires numerical optimization.

Here are some further related details.
\begin{enumerate}
\item Assumption \ref{hm:A_1} holds if 
\begin{enumerate}
    \item $\Psi(\gam)$ is $\mcc^4$-class;
    \item $\inf \Theta_{\sig^2} + \inf\Theta_{\sig_\ep^2} >0$.
\end{enumerate}
It may happen that $\inf_\gam \lam_{\min}(\Psi(\gam))=0$.

\item Assumption \ref{hm:A_4} holds if
\begin{enumerate}
    \item $\ds{\liminf_N \inf_v \lam_{\min}\left(\frac1N \sumi X_i^\top \Sig_i(v)^{-1}X_i\right)>0}$;
    \item There exists an $i_0\ge 1$ for which $\Sig_i(v)\ne \Sig_i(v_0)$ whenever $v\ne v _0$;
    \item $\ds{\liminf_N \lam_{\min}\left(\frac1N \sumi 
    \tr\left(\Sig_i(v_0)^{-1}\p_v \Sig_i(v_0) \Sig_i(v_0)^{-1}\p_v \Sig_i(v_0)\right)
    \right)>0}$.
\end{enumerate}
Here, we used the fact that the inequality $\log|A|-\log|B|-\tr(B^{-1}A-I_a) \le 0$ holds for any $a\times a$ symmetric positive definite matrices $A$ and $B$ with the equality holding if and only if $A=B$.
The items (b) and (c) correspond to the two items mentioned just after Assumption \ref{hm:A_4}.
\end{enumerate}
Assumption \ref{hm:A_2} (moment conditions) is needed as it is. 
Also, as already mentioned in Remark \ref{hm:rem_LLN-with-rate}, the convergences in Assumptions \ref{hm:A-S.lim} and the convergences at the $N^{1/2}$-rate required in Assumptions \ref{hm:A_3} and \ref{hm:A-Gam} and also in \eqref{ti: QLA.ass2.2} and \eqref{ti:QLA.ass2.6} are not straightforward to verify in the present unbalanced sampling framework.

%%%%%
%%%%%
\rrev{
\section{Data analysis}
\label{sec:simulations}
}
\subsection{Numerical experiments}
\label{subsec:simulations}
%\color{magenta}
%\begin{enumerate}
%    \item Joint \& stepwise, $N=500$, $13\le n_i\le 15$
%    \item Observing Studentization (Corollaries \ref{hm:cor_jASN} and \ref{hm:cor_sASN}):
%    $$ \widetilde{Z}_N = (\widetilde{Z}_{k,N})_{k\le p} := \left(\widetilde{\Gam}_N^{-1} \widetilde{S}_N \widetilde{\Gam}_N^{-1}\right)^{-1/2}\widetilde{u}_N \cil N_p(0,I_p). $$
%    Then, observe the $N(0,1)$-fittings of $\widetilde{Z}_{k,N}$ separately. Or,
%    $$ |\widetilde{Z}_N|^2 \cil \chi^2(p). $$
%    
%    \item $L=1000$? mean and s.d., histogram, boxplot (of $\tes$, $\widetilde{\theta}_N$), QQ plot
%    \item \texttt{system.time} for joint and stepwise (1--3): hope the stepwise is much faster. In particular, what amount is required for step 2 compared with steps 1 and 3.
%\end{enumerate}
%\color{black}

We performed numerical experiments to evaluate the asymptotic normality of the GQMLE under the non-Gaussian distribution of the longitudinal data $(Y_i)_{i \geq 1}$ and to evaluate the differences between the joint and stepwise GQMLEs. We assumed a scenario where the random-effect distribution does not follow the Gaussian distribution. For the evaluation of differences between the joint and stepwise GQMLEs, we confirmed the bias and computational load for the estimates. Our numerical experiment was conducted with the R software. 

For the numerical experiment, we generated the longitudinal data $(Y_i(t_{ij}))$ for $i=1, \dots, N$ and $j=1, \dots, n_i$ from the model: % \eqref{hm:target.model}
\begin{align}
Y_i(t_{ij}) &= \beta_1 + \beta_2\times t_{ij} + \beta_3\times g_i + b_{i} + W_{ij} + \ep_i(t_{ij}) \nn
\end{align}
with the explanatory variables $X_i(t_{ij})=(1, t_{ij}, g_i)$ and $Z_i(t_{ij})=1$. Here, $g_i$ denotes a dummy variable representing two hypothetical treatment groups (i.e., treatment or control group), which were generated from a binomial distribution ($p=0.5$). \rev{We used the Wiener process as L\'{e}vy noise in all our numerical experiments.} The random system-noise variable $(W_i(t_{ij}))_{j=1, \dots, n_i}$ followed a multivariate Gaussian distribution with the mean zero vector and the covariance matrix $H_i(1.30, 0.40^2)$. The true fixed-effect parameter was given as $(\beta_1,\beta_2,\beta_3)=(2.0, -1.0, 0.5)$. The number of time points $n_i$ was obtained from the integer part of Uniform(15,20)-random number, and the measurement time points $t_{i1}, \dots, t_{in_i}$ were randomly selected from $\{1,2, \dots, 20\}$ for each individual. The measurement error vector $(\ep_i(t_{ij}))_{j=1, \dots, n_i}$ followed a multivariate Gaussian distribution with the zero-mean vector and the diagonal covariance matrix $0.5^2 \times I_{n_i}$. The random effect followed a variance-gamma (VG) distribution whose density is given by
\begin{align}
% f(x; a_1, a_2, a_3, a_4) &= 
x &\mapsto 
\frac{2a_1^{a_1}(2a_1+a_4^2 a_3^{-1})^{\frac{1}{2}-a_1}}{\sqrt{2\pi a_3}\,\Gamma(a_1)} 
\nn\\
&{}\qquad 
\times \frac{K_{a_1-\frac{1}{2}}\left(\sqrt{Q(x;a_2,a_3)(2a_1+a_4^2 a_3^{-1})}\right) e^{(x-a_2)a_3^{-1} a_4}}{\left(\sqrt{Q(x;a_2,a_3)(2a_1+a_4^2a_3^{-1})}\right)^{\frac{1}{2}-a_1}},
\nn
\end{align}
where $K_{a}(\cdot)$ is the modified Bessel function of the third kind, $Q(x;a_2,a_3)=(x-a_2)^2/a_3$. 
% \begin{align}
% f(x; a_1, a_2, a_3, a_4) = \frac{2a_1^{a_1}(2a_1+a_4^2 a_3^{-1})^{\frac{1}{2}-a_1}}{\sqrt{2\pi a_3}\,\Gamma(a_1)} \times \frac{K_{a_1-\frac{1}{2}}\left(\sqrt{Q(x)(2a_1+a_4^2 a_3^{-1})}\right) e^{(x-a_2)a_3^{-1} a_4}}{\left(\sqrt{Q(x)(2a_1+a_4^2a_3^{-1})}\right)^{\frac{1}{2}-a_1}},
% \nn
% \end{align}
% where $K_{a}(\cdot)$ is the modified Bessel function of the third kind, $Q(x)=(x-a_2)^2/a_3$. 
This probability density function is asymmetric and has a heavier tail than the Gaussian distribution. We generated the VG-random numbers by using %the `VG' constructor function in 
the R-package \texttt{ghyp}. 
The true parameters were given as $(a_1, a_2, a_3, a_4)=(3,-3,0.1,3)$, then the mean and variance of the random effect were $0$ and $\sig_b^2:=3.01$, respectively. 
Thus, the values of the true parameters are summarized as follows:
$\theta = \left(\beta_1,\beta_2,\beta_3, \sig_b^2, \lam,\sig^2,\sig_\ep^2 \right)
= (2.0, -1.0, 0.5, 3.01, 1.3, 0.4^2, 0.5^2)$. 
% \begin{align}
% \theta = \left(\beta_1,\beta_2,\beta_3, \sig_b^2, \lam,\sig^2,\sig_\ep^2 \right)
% = (2.0, -1.0, 0.5, 3.01, 1.3, 0.4^2, 0.5^2).
% \nonumber
% \end{align}
\rev{The number of parameters to be estimated is seven in total: three fixed-effect parameters, one random-effect variance parameter, two system-noise parameters, and one measurement error variance parameter. Based on the true parameters, we generated 1000 Monte Carlo data sets for sample sizes ($N$) of 100, 300, 500, and 1000, respectively.}
We used the built-in \textbf{optim} function to numerically optimize the joint GQLF and stepwise GQLF (Stage 2). The Nelder-Mead method was applied as the optimization algorithm. \rev{The joint GQMLE was obtained by optimizing all parameters simultaneously, while the stepwise GQMLE was obtained stepwise by the method of Stages 1 to 3 in Subsection \ref{ti:subsec_stepwise_construction}.}

We now discuss the results of the numerical experiments. For the computation time of the joint and stepwise estimates, Table \ref{time.tab} shows summary statistics, and Figure \ref{time.fig} shows the box plots. %within 1000 seconds. 
The computation time for obtaining the stepwise GQMLE is much shorter than that for the joint GQMLE \rrev{for all sample size settings}. \rev{In this numerical experiment setup, as described in Section \ref{hm:sec_linear.case}, the stepwise GQMLE (estimates of fixed effect parameters) for Stage 1 and 3 are explicitly given, which is not the case for the joint GQMLE. Therefore, in the stepwise GQMLE, only the variance-covariance parameters of the longitudinal response variable are numerically optimized, resulting in a much shorter time than the joint GQMLE.} 
Table \ref{bias.tab} shows the means and standard deviations of the biases, that is, the differences between each parameter and the true parameters for $1000$ iterations. 
% the bias for each parameter of the two estimates in $N=500$ and $1000$, respectively; 
For the parameters regarding the fixed-effect, the variance of the random-effect, and the variance of the measurement error, the results are similar for both joint and stepwise GQMLEs, \rev{even with small sample sizes}. In contrast, for the two parameters in the system-noise, $\lam$ and $\sig$, the biases of the stepwise GQMLE are greater than those of the joint GQMLE, \rev{especially when the sample size is small}. As the sample size increases, the biases of the stepwise GQMLE become smaller, so a larger sample size seems necessary to obtain estimates that are less different from the true parameters. Figures \ref{histogram.fig} and \ref{qqplot.fig} show histograms and normal quantile-quantile plots (Q-Q plots) for the joint and stepwise GQMLEs. From these figures, the standard normal approximation seems to hold for both estimators well.
% \tcb{
% We observed the following.
% \begin{itemize}
%     \item For the parameters regarding the fixed-effect, the variance of the random-effect, and the variance of the measurement error, the results are similar for both joint and stepwise GQMLEs, \rev{even with small sample sizes}. 
%     \item In contrast, for the two parameters in the system-noise, $\lam$ and $\sig$, the biases of the stepwise GQMLE are greater than that of the joint GQMLE, \rev{especially when the sample size is small}. 
%     \item As the sample size $N$ increases, the biases of the stepwise GQMLE become smaller, so a larger sample size seems necessary to obtain estimates that are less different from the true parameters. 
%     \item Figures \ref{histogram.fig} and \ref{qqplot.fig} show histograms and normal quantile-quantile plots (Q-Q plots) for the joint and stepwise GQMLEs. From these figures, the standard normal approximation seems to hold for both estimators well.
% \end{itemize}
% }

\begin{center}
% \begin{threeparttable}[h]
\begin{table}[h]
\centering
\caption{Summary statistics of the computation time (seconds) for calculating the joint and stepwise GQMLEs for $1000$ iterations; SD means Standard deviation.}
\label{time.tab}
\begin{footnotesize}
\rrev{
\begin{tabular}{llcc}
\toprule
\multicolumn{1}{l}{Sample size ($N$)} & \multicolumn{1}{c}{Statistics} & \multicolumn{1}{c}{Joint} & \multicolumn{1}{c}{Stepwise}
\\
\hline 
$N = 100$ & Mean (SD) & 120.18 (28.13) & 43.47 (9.76) \\
        & Min, Max & 52.23, 210.36 & 19.64, 90.47 \\
        \hdashline 
$N = 300$ & Mean (SD) & 404.70 (86.57) & 123.13 (26.55) \\
        & Min, Max & 209.67, 718.09 & 62.64, 243.60 \\
        \hdashline 
$N = 500$ & Mean (SD) & 869.27 (3245.50) & 201.34 (60.80) \\
        & Min, Max & 371.51, 86005.75 & 107.75, 1607.64 \\
        \hdashline 
$N = 1000$ & Mean (SD) & 763.71 (1027.20) & 204.99 (243.49) \\
        & Min, Max & 395.19, 32547.95 & 105.13, 7679.54 \\
\hline
\end{tabular}
}
\end{footnotesize}
% \begin{tablenotes}
% \item Unit: second 
% \item SD = Standard deviation
% \end{tablenotes}
\end{table}
% \end{threeparttable}
\end{center}

\begin{figure}[tbh]
\centering
\includegraphics[width=10cm, height=7cm]{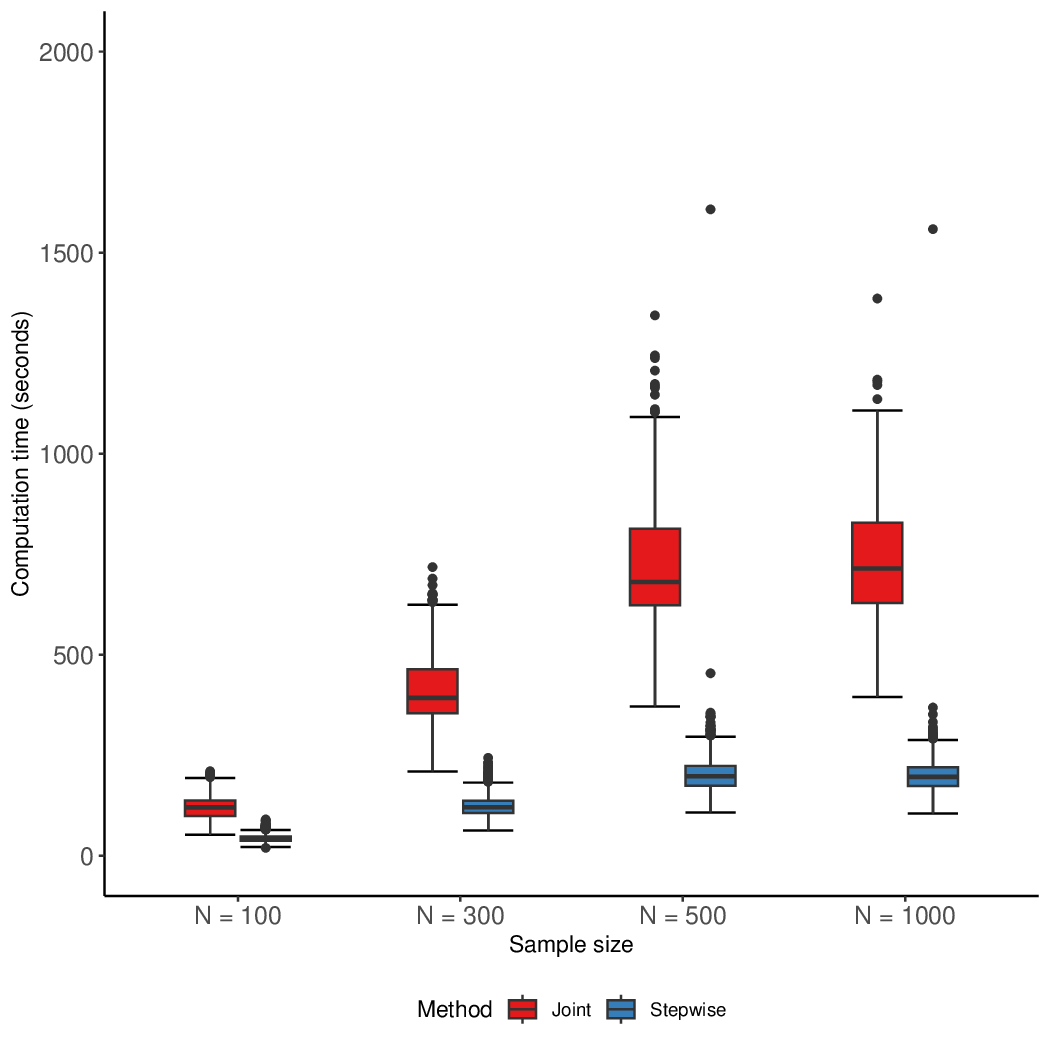}
\rrev{\caption{Box plots of the computation times (seconds) for calculating the joint GQMLE and the stepwise GQMLE for $1000$ iterations.}
\label{time.fig}
}
\end{figure}

\medskip

\begin{center}
% \begin{threeparttable}[h]
\begin{table}[h]
\centering
\caption{
The mean bias and the standard deviation (SD) of the joint and stepwise GQMLEs for $1000$ iterations.
% The mean and the standard deviation (SD) of the difference between each parameter for the joint GQMLE and the stepwise GQMLE and the true parameters for $1000$ iterations
}
\label{bias.tab}
\begin{footnotesize}
\begin{tabular}{rrrrrr}
\toprule
\multicolumn{1}{c}{} & \multicolumn{4}{c}{Joint GQMLE} \\
\cmidrule(lr){2-5}
\multicolumn{1}{c}{Parameter} & \multicolumn{1}{c}{$N = 100$} & \multicolumn{1}{c}{$N = 300$} & \multicolumn{1}{c}{$N = 500$} & \multicolumn{1}{c}{$N = 1000$}
\\
\hline 
\multicolumn{1}{c}{$\beta_1$} & 0.003 (0.259) & 0.001 (0.152) & -0.002 (0.118) & -0.001 (0.086) \\
\multicolumn{1}{c}{$\beta_2$} & 0.000 (0.008) & 0.000 (0.004) & 0.000 (0.003) & 0.000 (0.002) \\
\multicolumn{1}{c}{$\beta_3$} & 0.002 (0.372) & 0.000 (0.215) & -0.001 (0.164) & 0.001 (0.117) \\
\multicolumn{1}{c}{$\gamma$} & -0.045 (0.661) & -0.011 (0.371) & -0.002 (0.283) & -0.010 (0.206) \\
\multicolumn{1}{c}{$\lambda$} & -0.083 (0.756) & 0.019 (1.144) & -0.001 (0.725) & 0.043 (0.733) \\
\multicolumn{1}{c}{$\sigma$} & -0.023 (0.226) & 0.006 (0.340) & 0.000 (0.215) & 0.013 (0.217) \\
\multicolumn{1}{c}{$\sigma_{\epsilon}$} & 0.004 (0.015) & 0.002 (0.009) & 0.002 (0.008) & 0.001 (0.006) \\
\hline
\end{tabular}

\bigskip

\begin{tabular}{rrrrrr}
\toprule
\multicolumn{1}{c}{} & \multicolumn{4}{c}{Stepwise GQMLE} \\
\cmidrule(lr){2-5}
\multicolumn{1}{c}{Parameter} & \multicolumn{1}{c}{$N = 100$} & \multicolumn{1}{c}{$N = 300$} & \multicolumn{1}{c}{$N = 500$} & \multicolumn{1}{c}{$N = 1000$}
\\
\hline 
\multicolumn{1}{c}{$\beta_1$} & 0.004 (0.255) & 0.001 (0.149) & 0.000 (0.116) & -0.001 (0.083) \\
\multicolumn{1}{c}{$\beta_2$} & 0.000 (0.008) & 0.000 (0.004) & 0.000 (0.003) & 0.000 (0.002) \\
\multicolumn{1}{c}{$\beta_3$} & -0.001 (0.367) & -0.004 (0.212) & -0.003 (0.160) & -0.001 (0.115) \\
\multicolumn{1}{c}{$\gamma$} & -0.061 (0.650) & -0.019 (0.367) & -0.008 (0.283) & -0.015 (0.204) \\
\multicolumn{1}{c}{$\lambda$} & 2.278 (14.152) & 1.433 (15.511) & 1.110 (10.065) & 0.250 (1.074) \\
\multicolumn{1}{c}{$\sigma$} & 0.687 (4.420) & 0.427 (4.600) & 0.331 (2.994) & 0.074 (0.320) \\
\multicolumn{1}{c}{$\sigma_{\epsilon}$} & -0.004 (0.017) & -0.002 (0.011) & -0.002 (0.009) & -0.001 (0.006) \\
\hline
\end{tabular}
\end{footnotesize}
\end{table}
% \end{threeparttable}
\end{center}

\begin{figure}[tbh]
\centering
\includegraphics[width=12cm]{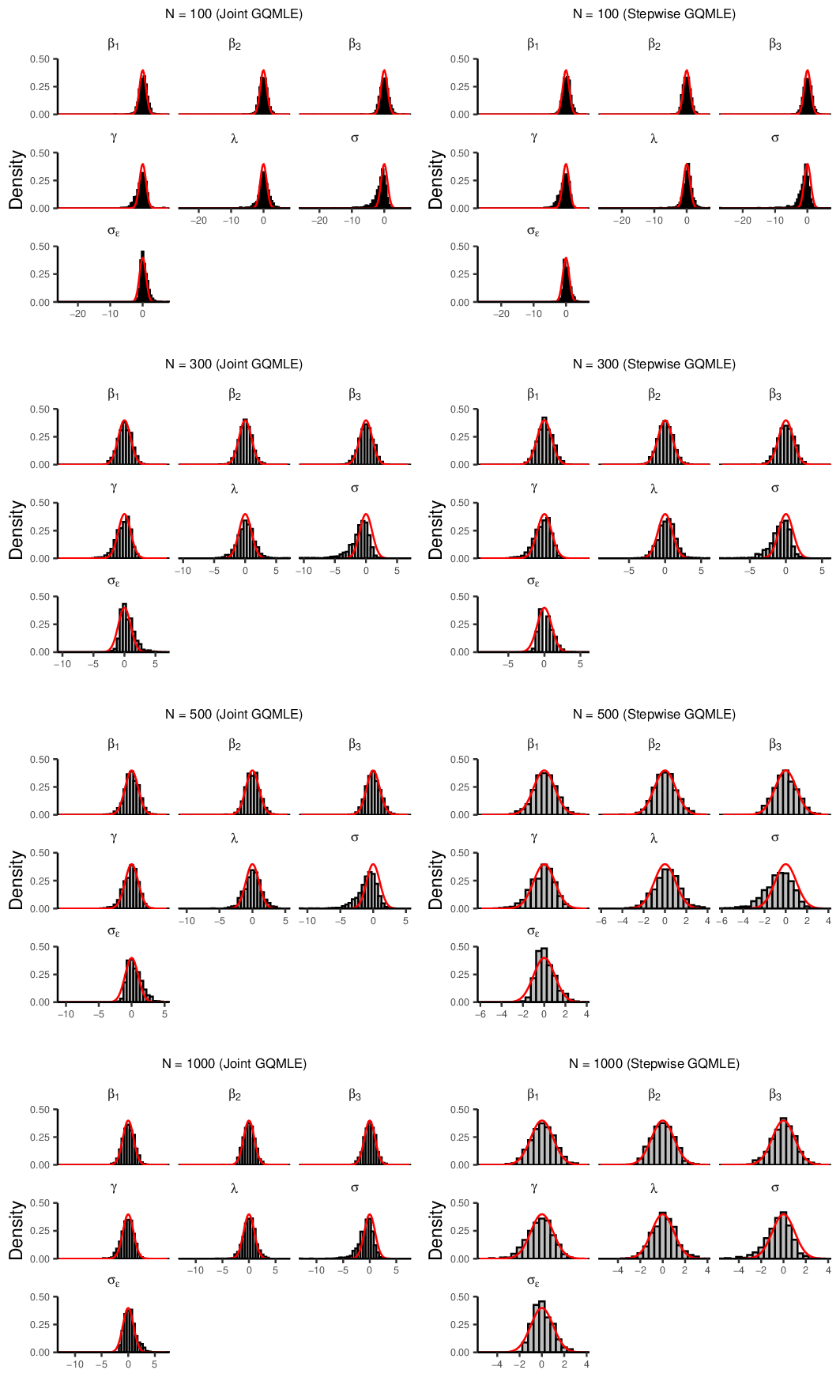}
\rrev{\caption{Histograms of the studentized joint and stepwise GQMLEs and probability density function of standard Gaussian distribution (red curve)}
\label{histogram.fig}
}
\end{figure}

\begin{figure}[tbh]
\centering
\includegraphics[width=12cm]{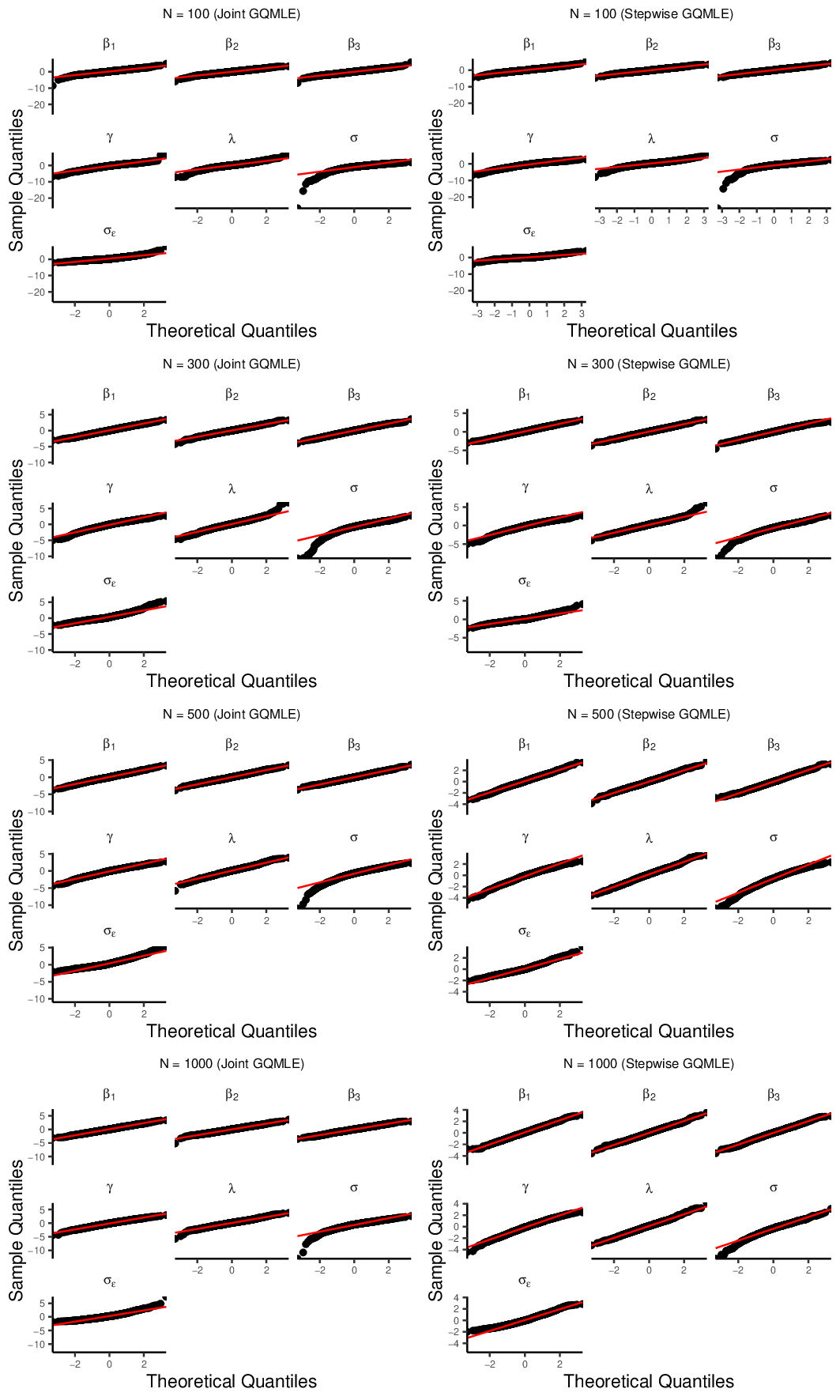}
\rrev{\caption{Normal Q-Q plots of the studentized joint and stepwise GQMLEs}
\label{qqplot.fig}
}
\end{figure}

\newpage

\rrev{
\subsection{Application to the primary biliary cholangitis data}
\label{subsec:RWD}
We applied the LME model with the intOU process \eqref{hm:target.model} to a dataset of primary biliary cholangitis (PBC) from the R package \textbf{JM} \cite{JMPackage}. The PBC data was obtained from the Mayo Clinic trial between 1974 and 1984 \cite{PBC1994}. Of the 312 patients with PBC, 158 were randomly assigned to the D-penicillamine group and 154 to the placebo group. In addition to baseline demographic characteristics, biomarkers including serum bilirubin and albumin were repeatedly measured during the follow-up period. We evaluated the impact of the IOU process on parameter estimates and model performance in the LME model. In this model, the log-ratio to baseline in serum bilirubin was defined as the response variable, with fixed effects for baseline bilirubin, treatment, time, and the treatment-by-time interaction, and random intercepts and slopes for time. A total of 285 patients who had serum bilirubin data available after randomization were included in the analysis. Figure \ref{SerumBili.fig} shows the observed trajectories by individual. The number of post-randomization serum bilirubin measurements averaged 5.7, ranging from 1 to 15. The dataset was unbalanced, as both the number of measurements and the measurement intervals varied by individual. 
}

\rrev{
We compared the LME models with and without the intOU process. Furthermore, both the joint and stepwise Gaussian quasi-likelihood approaches were applied to the LME model with the intOU process. Table \ref{ti:results_RQD} presents the parameter estimates, standard errors, and the BIC (see \eqref{ti:BIC}) for each model. Regardless of the inclusion of the intOU process, the fixed-effects parameters yield similar estimates. This result is expected, because incorporating the intOU process only alters the variance structure. The model including the intOU process provides a better fit in terms of BIC. Furthermore, the joint and stepwise GQMLEs exhibit no substantial difference in the model with the intOU process.
}

\begin{figure}[tbh]
\centering
\includegraphics[width=12cm]{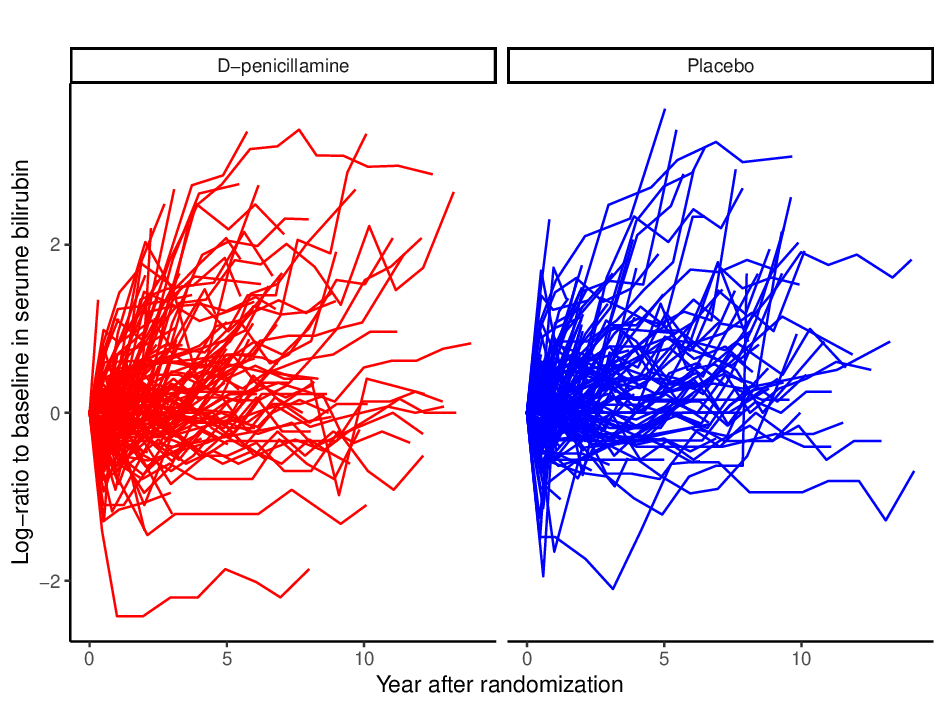}
\rrev{\caption{Spaghetti plot of log-ratio to baseline in serum bilirubin}
\label{SerumBili.fig}
}
\end{figure}

\begin{center}
\begin{table}[h]
\centering
\caption{
Parameter estimates and BIC.
}
\label{ti:results_RQD}
\begin{footnotesize}
\rrev{
\begin{tabular}{rrrr}
\toprule
\multicolumn{1}{l}{} & \multicolumn{1}{c}{LME model} & \multicolumn{2}{c}{LME model}\\
\multicolumn{1}{l}{} & \multicolumn{1}{c}{without the intOU process} & \multicolumn{2}{c}{with the intOU process}\\
\cmidrule(lr){3-4}
\multicolumn{1}{l}{} & \multicolumn{1}{c}{} & \multicolumn{1}{c}{Joint} & \multicolumn{1}{c}{Stepwise} 
\\
\hline 
\multicolumn{1}{l}{Estimates ${}^{1}$ (S.E. ${}^{2}$)} & \multicolumn{3}{c}{} \\
\multicolumn{1}{r}{$\beta_1$: Intercept} & \multicolumn{1}{c}{-0.01869 (0.04710)} & \multicolumn{1}{c}{-0.01738 (0.04644)} & \multicolumn{1}{c}{-0.02550 (0.04937)} \\
\multicolumn{1}{r}{$\beta_2$: Baseline value} & \multicolumn{1}{c}{-0.01470 (0.00746)} & \multicolumn{1}{c}{-0.01427 (0.00689)} & \multicolumn{1}{c}{-0.01202 (0.00693)} \\
\multicolumn{1}{r}{$\beta_3$: Year} & \multicolumn{1}{c}{0.17397 (0.01869)} & \multicolumn{1}{c}{0.17066 (0.01951)} & \multicolumn{1}{c}{0.17293 (0.02116)} \\
\multicolumn{1}{r}{$\beta_4$: Group} & \multicolumn{1}{c}{-0.09523 (0.05910)} & \multicolumn{1}{c}{-0.09174 (0.05832)} & \multicolumn{1}{c}{-0.10544 (0.06260)} \\
\multicolumn{1}{r}{$\beta_5$: Group $\times$ Year} & \multicolumn{1}{c}{0.00727 (0.02604)} & \multicolumn{1}{c}{0.00039 (0.02728)} & \multicolumn{1}{c}{-0.00771 (0.02959)} \\
\multicolumn{1}{r}{$\gam_1$} & \multicolumn{1}{c}{0.41188 (0.02972)} & \multicolumn{1}{c}{0.39435 (0.03327)} & \multicolumn{1}{c}{0.43247 (0.03099)} \\
\multicolumn{1}{r}{$\gam_2$} & \multicolumn{1}{c}{0.00452 (0.00218)} & \multicolumn{1}{c}{-0.00876 (0.00609)} & \multicolumn{1}{c}{-0.03189 (0.00686)} \\
\multicolumn{1}{r}{$\gam_3$} & \multicolumn{1}{c}{0.17040 (0.00103)} & \multicolumn{1}{c}{0.04138 (0.14952)} & \multicolumn{1}{c}{0.00996 (0.67876)} \\
\multicolumn{1}{r}{$\lam$} & \multicolumn{1}{c}{--- (---)} & \multicolumn{1}{c}{0.33247 (0.15054)} & \multicolumn{1}{c}{0.30953 (0.11820)} \\
\multicolumn{1}{r}{$\sig$} & \multicolumn{1}{c}{--- (---)} & \multicolumn{1}{c}{0.19992 (0.03611)} & \multicolumn{1}{c}{0.21532 (0.03060)} \\
\multicolumn{1}{r}{$\sig_\ep$} & \multicolumn{1}{c}{0.33472 (0.01356)} & \multicolumn{1}{c}{0.26717 (0.01307)} & \multicolumn{1}{c}{0.26564 (0.01297)} \\
\hdashline
\multicolumn{1}{l}{BIC} & \multicolumn{1}{c}{2171.782} & \multicolumn{1}{c}{1984.869} & \multicolumn{1}{c}{1992.382} \\
\hline
\end{tabular}
  \begin{minipage}{\linewidth}
    \vspace{0.2cm}
    \footnotemark[1] Parameters: $\beta_1, \dots, \beta_5$ denote the fixed-effect parameters; 
    $\gam_1$ and $\gam_3$ denote the variance parameters for the random intercept and slope, and $\gam_2$ denotes their covariance; $\lam$ and $\sig$ denote the autoregression and scale parameters of the intOU process; finally, $\sig_\ep$ denotes the variance parameter for measurement error.\\
    \footnotemark[2] S.E. = Standard errors obtained by the consistent estimators of asymptotic variances: \eqref{hm:jASN} for the joint GQMLE, \eqref{hm:sASN} for the stepwise GQMLE.
  \end{minipage}
} % <-- \rrev
\end{footnotesize}
\end{table}
\end{center}

%%%%%
\section{Concluding remarks}
\label{sec:discussion}

In this paper, we considered the asymptotic behavior of the joint and stepwise GQMLE for the class of possibly non-Gaussian LME models. We proved that both estimators have asymptotic normality with the same asymptotic covariance matrix and the tail-probability estimate. Moreover, we showed the quantitative difference in the second-order terms of the joint and stepwise GQMLEs: the equation \eqref{hm:def_G_2N} in Theorem \ref{ti:QLA.thm1} and the equations \eqref{ti:def_G_2N_beta} and \eqref{ti:def_G_2N_nu} in Theorem \ref{ti:QLA.thm2}. This should be informative in studying the cAIC, which involves the second-order stochastic expansion of the estimator. 
We also note that, as we mentioned in \cite[Remark 2.5]{ImaMasTaj24}, instead of the intOU process we could consider the fractional Brownian motion to model the system noise for each individual.

The numerical experiments demonstrated that both the joint and stepwise GQMLEs exhibit competitive performance consistent with asymptotic normality. In particular, the computation time for the stepwise GQMLE is significantly shorter than that for the joint GQMLE. On the other hand, the variation of the stepwise GQMLE tends to be larger than that of the joint GQMLE, particularly in small-sample settings. 
\rrev{
Our analysis of the PBC data suggests that the model incorporating the intOU process provides a better fit than the model without it. By flexibly capturing complex variance structures, the intOU process is expected to enhance performance for unbalanced datasets, such as the PBC data, where the number of measurements and measurement intervals vary across individuals.
}

%%%%%
%%%%%
\bigskip

\noindent
\textbf{Acknowledgements.} 
This work was partially supported by JST CREST Grant Number JPMJCR2115 and JSPS KAKENHI Grant Numbers 23K22410 (22H01139), Japan.

\bigskip

\bibliographystyle{abbrv} % plain,jplain,abbrv,unsrt,alpha,apalike
\bibliography{imamura_bibs}

\begin{thebibliography}{10}

\bibitem{AdaFou03}
R.~A. Adams and J.~J.~F. Fournier.
\newblock {\em Sobolev spaces}, volume 140 of {\em Pure and Applied Mathematics
  (Amsterdam)}.
\newblock Elsevier/Academic Press, Amsterdam, second edition, 2003.

\bibitem{asar2020linear}
{\"O}.~Asar, D.~Bolin, P.~J. Diggle, and J.~Wallin.
\newblock Linear mixed effects models for non-{Gaussian} continuous repeated
  measurement data.
\newblock {\em Journal of the Royal Statistical Society Series C: Applied
  Statistics}, 69(5):1015--1065, 2020.

\bibitem{BosTayLaw1998}
W.~J. Boscardin, J.~M. Taylor, and N.~Law.
\newblock Longitudinal models for {AIDS} marker data.
\newblock {\em Statistical Methods in Medical Research}, 7(1):13--27, 1998.

\bibitem{EguMas18}
S.~Eguchi and H.~Masuda.
\newblock Schwarz type model comparison for {LAQ} models.
\newblock {\em Bernoulli}, 24(3):2278--2327, 2018.

\bibitem{EguMas24}
S.~Eguchi and H.~Masuda.
\newblock Gaussian quasi-information criteria for ergodic {L}\'{e}vy driven
  {SDE}.
\newblock {\em Ann. Inst. Statist. Math.}, 76(1):111--157, 2024.

\bibitem{HugKenSteTil17}
R.~A. Hughes, M.~G. Kenward, J.~A.~C. Sterne, and K.~Tilling.
\newblock Estimation of the linear mixed integrated {O}rnstein-{U}hlenbeck
  model.
\newblock {\em J. Stat. Comput. Simul.}, 87(8):1541--1558, 2017.

\bibitem{ImaMasTaj24}
T.~Imamura, H.~Masuda, and H.~Tajima.
\newblock On local likelihood asymptotics for {G}aussian mixed-effects model
  with system noise.
\newblock {\em Statist. Probab. Lett.}, 208:Paper No. 110074, 5, 2024.

\bibitem{KesSchWef01}
M.~Kessler, A.~Schick, and W.~Wefelmeyer.
\newblock The information in the marginal law of a {M}arkov chain.
\newblock {\em Bernoulli}, 7(2):243--266, 2001.

\bibitem{Kub11}
T.~Kubokawa.
\newblock Conditional and unconditional methods for selecting variables in
  linear mixed models.
\newblock {\em J. Multivariate Anal.}, 102(3):641--660, 2011.

\bibitem{LairdWare1982}
N.~M. Laird and J.~H. Ware.
\newblock Random-effects models for longitudinal data.
\newblock {\em Biometrics}, pages 963--974, 1982.

\bibitem{MagNeu1979}
J.~R. Magnus and H.~Neudecker.
\newblock {The Commutation Matrix: Some Properties and Applications}.
\newblock {\em The Annals of Statistics}, 7(2):381 -- 394, 1979.

\bibitem{Mas04}
H.~Masuda.
\newblock On multidimensional {O}rnstein-{U}hlenbeck processes driven by a
  general {L}\'evy process.
\newblock {\em Bernoulli}, 10(1):97--120, 2004.

\bibitem{Mas07}
H.~Masuda.
\newblock Ergodicity and exponential {$\beta$}-mixing bounds for
  multidimensional diffusions with jumps.
\newblock {\em Stochastic Process. Appl.}, 117(1):35--56, 2007.

\bibitem{MasMerUeh24}
H.~Masuda, L.~Mercuri, and Y.~Uehara.
\newblock Quasi-likelihood analysis for student-l{\'e}vy regression.
\newblock {\em Stat. Inference Stoch. Process.}, 27(3):761--794, 2024.

\bibitem{McCNeu2011}
C.~E. McCulloch and J.~M. Neuhaus.
\newblock Misspecifying the shape of a random effects distribution: Why getting
  it wrong may not matter.
\newblock {\em Statistical Science}, 26:388--402, 2011.

\bibitem{McCNeu11}
C.~E. McCulloch and J.~M. Neuhaus.
\newblock Prediction of random effects in linear and generalized linear models
  under model misspecification.
\newblock {\em Biometrics}, 67(1):270--279, 2011.

\bibitem{MulSceWel13}
S.~M\"{u}ller, J.~L. Scealy, and A.~H. Welsh.
\newblock Model selection in linear mixed models.
\newblock {\em Statist. Sci.}, 28(2):135--167, 2013.

\bibitem{PBC1994}
P.~A. Murtaugh, E.~R. Dickson, G.~M. Van~Dam, M.~Malinchoc, P.~M. Grambsch,
  A.~L. Langworthy, and C.~H. Gips.
\newblock Primary biliary cirrhosis: Prediction of short-term survival based on
  repeated patient visits.
\newblock {\em Hepatology}, 20(1):126--134, 1994.

\bibitem{JMPackage}
D.~Rizopoulos.
\newblock Jm: An r package for the joint modelling of longitudinal and
  time-to-event data.
\newblock {\em Journal of Statistical Software}, 35(9):1--33, 2010.

\bibitem{TayCumSy94}
J.~M.~G. Taylor, W.~G. Cumberland, and J.~P. Sy.
\newblock A stochastic model for analysis of longitudinal aids data.
\newblock {\em Journal of the American Statistical Association},
  89(427):727--736, 1994.

\bibitem{VaiBla05}
F.~Vaida and S.~Blanchard.
\newblock Conditional {A}kaike information for mixed-effects models.
\newblock {\em Biometrika}, 92(2):351--370, 2005.

\bibitem{Yos11}
N.~Yoshida.
\newblock Polynomial type large deviation inequalities and quasi-likelihood
  analysis for stochastic differential equations.
\newblock {\em Ann. Inst. Statist. Math.}, 63(3):431--479, 2011.

\end{thebibliography}

% \clearpage
% \appendix
% \input{HM_notes.tex}

\end{document}